\documentclass[a4paper,11pt]{article}
\usepackage{amsthm,amsmath}
\usepackage{amssymb}
\RequirePackage[numbers]{natbib}
\RequirePackage{hyperref}
\usepackage{mathtools}
\usepackage{graphicx}
\usepackage{color}
\usepackage{tabularx}
\usepackage{float}
\usepackage{array}
\usepackage{multirow}

\newcommand{\mathd}{\mathrm{d}}
\newcommand{\mathe}{\mathrm{e}}

\newcommand{\esp}{\mathbb{E}}

\numberwithin{equation}{section}
\theoremstyle{plain}
\newtheorem{definition}{Definition}[section]
\newtheorem{proposition}[definition]{Proposition}
\newtheorem{lemma}[definition]{Lemma}
\newtheorem{theorem}[definition]{Theorem}
\newtheorem{corollary}[definition]{Corollary}
\newtheorem{remark}[definition]{Remark}

\newtheorem{assum}[definition]{Assumption}

\usepackage{graphicx}

\begin{document}
\title{Generalized Pareto Regression Trees for extreme events analysis}
\author{S\'ebastien F\textsc{arkas}$^1$, Antoine H\textsc{eranval}$^{1,2}$,  Olivier L\textsc{opez}$^1$, \\and  Maud T\textsc{homas}$^1$\\
 }

\date{}

\maketitle

\begin{center} 
{\small$^1$ Sorbonne Universit\'e, CNRS, Laboratoire de Probabilit\'es, Statistique et Mod\'elisation, \\LPSM, 4 place Jussieu, F-75005 Paris, France,\\
$^2$ Mission Risques Naturels, 1 rue Jules Lefebvre 75009 Paris, France\\
E-mails : sebastien.farkas@sorbonne-universite.fr, \\
antoine.heranval@mrn.asso.fr, \\maud.thomas@sorbonne-universite.fr, \\olivier.lopez@sorbonne- universite.fr}
\end{center}

\begin{abstract}
In this paper, we provide finite sample results to assess the consistency of Generalized Pareto regression trees, as tools to perform extreme value regression. The results that we provide are obtained from concentration inequalities, and are valid for a finite sample size, taking into account a misspecification bias that arises from the use of a ``Peaks over Threshold'' approach. The properties that we derive also legitimate the pruning strategies (i.e. the model selection rules) used to select a proper tree that achieves compromise between bias and variance. The methodology is illustrated through a simulation study, and a real data application in insurance against natural disasters.
\end{abstract}


 \section{Introduction}

Extreme value theory (EVT) is the branch of statistics which has been developed and broadly used to handle extreme events, such as extreme floods, heat waves episodes or extreme financial losses \citep{katz2002statistics, embrechts2013modelling}. One of the key results behind the success of this approach was proved by Balkema and de Haan in  \citep{balkema1974residual}: they show that the tail of the distribution of a series of observations can be approximated by a parametric family of distributions, namely Generalized Pareto (GP) distributions. This property allows the statistician to find information from the largest observations of a random sample to extrapolate the tail. This yields the so-called Peaks over Threshold (PoT) method introduced in \citep{smith1984threshold} which consists in fitting a GP distribution to the excesses above some (high) suitably chosen threshold. In a regression framework, the parameters of this GP distribution depend on covariates reflecting the fact that different values of these covariates may result in a different tail behavior of the response variable. In this paper, we study the use of regression trees to perform GP regression on the excesses. This ensemble method, introduced by \citep{breiman}, determines clusters of similar tail behaviors depending on the value of the covariates, based on a recursive partition of the sample and simple model selection rules. In the present work, we provide theoretical results and empirical evidence on the consistency of such a procedure and of these selection rules. The result we provide are based on concentration inequalities, in order to hold for finite sample sizes. The main difficulty stands in the misspecification of the model and on handling the fact that the distributions are heavy tailed.

Tail regression is a challenging task. Several papers have been interested in extreme quantile regression, to name a few, in 2005, Chernozhukov \citep{chernozhukov2005extremal} and, in 2012, Wang et al. \citep{wang2012estimation} derive extreme quantile estimators assuming a linear form for the conditional quantile. In 2019, Gardes and Stupfler \citep{gardes2019integrated} and Velthoen et al. \citep{velthoen2019improving} use conditional intermediate-level quantiles to extrapolate above the threshold and deduce estimators for extreme conditional quantiles. Another approach is to model the parameters of the GP distribution of functions of the covariates e.g. as local polynomials \citep{beirlant2} or as generalized additive models \citep{ChavezDemoulin2015}. Very recently, in 2021, Velthoen et al. \citep{velthoen2021gradient} proposed to a gradient boosting procedure to estimate conditional GP distribution. Let us note that the nonparametric approaches rely on regularity assumptions on the way the tail of the distribution evolves with the covariates (which are required to be continuous through the use of kernel smoothing). A nice feature of the regression tree approach we consider in the present paper is its ability to handle several covariates which components may be either discrete or continuous. Moreover, this method is adapted to situations where the tail behavior is supposed to be significantly different depending on the characteristics, as for example it is the case in an application to cyber-insurance considered in a former paper, see \citep{farkas2021cyber}.

Regression trees, introduced by Breiman \citep{breiman} along with the CART algorithm (for Clustering And Regression Trees), are flexible tools to perform a regression and clustering task simultaneously. They have been used in various fields, including industry \citep{cartindustry}, geology (see e.g. \citep{cartgeology}), ecology (see e.g. \citep{carteco}), claim reserving in insurance \citep{lopez2016}. Through the iterative splitting algorithm used in CART, nonlinearities are introduced in the way the distribution is modeled, while furnishing an intelligible interpretation of the final classification of response variables. The splitting criterion---used to iteratively separate observations into clusters of similar behaviors---depends on the type of problems one is considering. While the standard CART algorithm relies on mean-squared criterion to perform mean-regression, alternative loss functions have been considered as in \citep{chaudhuri2002} for quantile regression, or in \citep{SU_MLRT} who used a log-likelihood based loss. Loh \citep{loh1,loh2} provide detailed descriptions of regression trees procedures and a review of their variants. In this paper, building on the Balkema and de Haan result, we use a GP log-likelihood loss, as in \citep{farkas2021cyber}, to perform extreme value regression. 

The rest of the paper is organized as follows. In Section \ref{sec_methods}, we introduce notations and describe the GP regression tree algorithm. Section \ref{sec_main} lists the main results of this paper, that is deviation bounds for the regression tree estimator for finite sample size, and consistency of the ``pruning'' (that is model selection) strategy. Empirical results are gathered in Section \ref{sec_simu}, which provides a simulation study, and a real data analysis in natural disaster insurance. Detailed proofs of the technical results are shown in the Appendix.

\section{Regression trees for extreme value analysis}
\label{sec_methods}

This section describes the estimation method (GP regression trees) that is considered in the paper. Some classical results in EVT are given in Section \ref{sec_evtreg} to motivate the GP approximation. Regression trees adapted to this context are described in Section \ref{sec_GPDtrees}. A short discussion on the advantage of this technique compared to competing approaches is developed in Section \ref{subsec:compete}.

\subsection{Extreme value theory and regression}
\label{sec_evtreg}

Let us consider independent and identically distributed observations $Y_1,Y_2,\ldots$  with an unknown survival function $\overline{F}$ (that is $\overline{F}(y)=P(Y_1>y)$). A natural way to define extreme events is to consider the values of $Y_i$ which have exceeded some high threshold $u$. The excesses above $u$ are then defined as the variables $Y_i -u$ given that $Y_i > u$. The asymptotic behavior of extreme events is characterized by the distribution of the excesses which is given by 
\[
\overline{F}_u(z)= P[Y_1 - u > z \mid Y_1> u] = \frac{\overline F(u+z)}{\overline F(u)}, \ z>0 \, . 
\]
In 1975, Pickands \citep{pickands} showed that, if $\overline F$ satisfies the following property
\begin{equation}\label{eq:rv}
\lim_{t \to \infty} \frac{\overline{F}(ty)}{\overline{F}(y)} = y^{-1/\gamma_0}, \, \forall y >0, 
\end{equation}
with $\gamma_0>0$, then 
\begin{equation}\label{eqplus}
\lim_{u \to \infty} \sup_{z>0} |\overline F_u(z) - \overline{H}_{\sigma_{0u},\gamma_0}(z) | = 0
\end{equation}
for some $\sigma_{0u}>0$ and $\overline{H}_{\sigma_{0u},\gamma_0}$ necessarily belongs to the Generalized Pareto (GP) distributions family which distribution function is of the form 
 $$
\overline{H}_{\sigma_{0u},\gamma_0}(z) = \left(1+ \gamma_0 \frac{z}{\sigma_{0u}}\right)^{-1/\gamma_0}, \ z>0 ,
$$
where $\sigma_{0u}>0$ is a scale parameter and $\gamma_0>0$ is a shape parameter, which reflects the heaviness of the tail distribution. Especially, if $\gamma_0 \in ]0;1[,$ the expectation of $Y_1$ is finite whereas if $\gamma_0\geq1$ the expectation of $Y_1$ is infinite. More details on these results can be found in e.g. \citep{Coles2001,beirlant2004statistics}.

In practice, the so-called Peaks over Threshold (PoT) method is widely used, see \citep{davison1990models,Coles2001}. It consists in choosing a high threshold $u$ and fitting a GP distribution on the excesses above that threshold $u$. The estimation of the parameters $\sigma_{0u}$ and $\gamma_0$ may be done by maximizing the GP likelihood. The choice of the threshold $u$ can be understood as a compromise between bias and variance: the smaller the threshold, the less valid the asymptotic approximation, leading to bias; on the other hand, a too high threshold will generate few excesses to fit the model, leading to high variance. The existing methods are mostly graphical, up to our knowledge, no automatic data-driven selection procedure is available. 

In the present paper, we consider a regression framework, that is that our goal to know the impact of some random covariates $\mathbf{X}$ on the tail of the distribution of a response variable $Y.$ The previous convergence results hold, but for quantities $\sigma_{0u},$ $\gamma_0$ and $u$ that may depend on $\mathbf{X}.$ More precisely, this means that, if we assume that $\gamma_0(\mathbf{x})>0$ for all $\mathbf{x}$ (which is the assumption that we will make throughout this paper), then (\ref{eq:rv}) becomes
\begin{equation}\label{eq:rv2}
\lim_{t \to \infty} \frac{\overline{F}(ty|\mathbf{x})}{\overline{F}(y|\mathbf{x})} = y^{-1/\gamma_0(\mathbf{x})}, \, \forall y >0, 
\end{equation}
where $\overline F(y|\mathbf{x})=\mathbb{P}(Y\geq y |\mathbf{X}=\mathbf{x}),$ see \citep{beirlant2004statistics} and references therein, and
 (\ref{eqplus}) becomes
\begin{equation}\label{eqplus2}
\lim_{u(\mathbf{x}) \to \infty} \sup_{z>0} |\overline F_{u(\mathbf{x})}(z  \mid \mathbf{x}) - \overline{H}_{\sigma_{0u(\mathbf x)}(\mathbf{x}),\gamma_0 (\mathbf x)}(z) | = 0. 
\end{equation}
where $\overline F_{u(\mathbf{x})}(z  \mid \mathbf{x})=P[Y-u(\mathbf{x}) > z \mid Y > u(\mathbf{x}),\mathbf{X}=\mathbf{x}].$

Suppose that we observe  $(Y_i,\mathbf{X}_i)_{1\leq i \leq n}$ a sample of $(Y,\mathbf{X}),$ where $\mathbf{X}$ belongs to a compact set $\mathcal{X}.$ Following the PoT approach, the estimation of the function $\gamma_0(\mathbf{X})$ and $\sigma_0(\mathbf{X})=\sigma_{0u(\mathbf{X})}$ can typically be done by fitting a regression model on the data points $(Y_i,\mathbf{X}_i)$ such that $Y_i$ exceeds a proper threshold $u(\mathbf{X}_i).$ More precisely, let us define \begin{eqnarray}
 \theta^*(\mathbf x)&=&\arg \max_{ \theta\in \Theta} \mathbb \esp [\phi(Y-u(\mathbf{X}),\theta)|\mathbf{X} =\mathbf{x},Y\geq u(\mathbf{x})] \nonumber \\ &=& \arg \max_{ \theta\in \Theta}\esp [\phi(Y-u(\mathbf{X}),\theta)\mathbf{1}_{Y\geq u(\mathbf{X})}|\mathbf{X}=\mathbf{x}], \label{m_est}
 \end{eqnarray}
 where $\theta=(\sigma,\gamma)^{\tau}$ (where $a^{\tau}$ denotes the transpose of a vector $a$) and $\phi$ is the GP log-likelihood function, that is
\[
\phi(z,\theta) = -\log(\sigma)-\left(\frac{1}{\gamma}+1\right)\log\left(1+\frac{\gamma z}{\sigma} \right).
\]
From (\ref{eqplus2}), $\theta^*(\mathbf{x})$ should be close to $\theta_0(\mathbf{x})=(\sigma_0(\mathbf{x}),\gamma_0(\mathbf{x}))^{\tau}$ for $u(\mathbf{x})$ large enough. Based on this idea, Beirlant (2004) \citep{beirlant2} proposed a nonparametric approximation of the loss function maximized by $\theta^*.$ This technique, based on local polynomials, requires continuity of the covariates and some smoothness assumptions on $\theta_0.$ On the other hand, parametric methods  \citep[][]{ChavezDemoulin2015,beirlant2003} have also been proposed, but relying on a stronger assumption on the shape of $\theta_0.$

In the next section, we introduce a regression tree approach which is adapted to both continuous and discrete variables, and that relies on few assumptions (since the estimated regression function $\theta_0$ does not need to be smooth).

\subsection{GPD regression trees}
\label{sec_GPDtrees}

Regression Trees are a convenient tool to capture heterogeneous behaviors in the data, see \citep{breiman}. These models aim at constituting classes of observations which have a relatively similar behavior in terms of the response variable $Y.$ These classes are defined by ``rules'', which affect an observation to one of these classes according to the values of its covariates $\mathbf{X}$. These rules are obtained from the data through the CART (Clustering And Regression Tree) algorithm, and the non-linearity of the procedure allows for an adaptation to the estimation of large classes of regression functions.

Fitting regression trees relies on a so-called ``growing phase'', described in our context in Section \ref{subsec:buildingProc}, which corresponds to the determination of these splitting rules. Section \ref{sec5} shows how an estimator of the regression function $\theta_0$ can be deduced from such a tree. The ``pruning step'', which can be understood as a model selection procedure, is described in Section \ref{subsec:pruning}.

\subsubsection{Growing step: construction of the maximal tree}
\label{subsec:buildingProc}

The CART algorithm consists in determining iteratively a set of  ``rules'' $\mathbf{x}=(x^{(1)},\ldots,x^{(d)})\rightarrow R_{j}(\mathbf{x})$ to split the data, aiming at optimizing some objective function (also referred to as splitting criterion). In our case, 
we want to approximate the criterion (\ref{m_est}), that is we are searching for a regression function $\widehat{\theta}(\mathbf{X})$ among some class such that $\sum_{i=1}^n \phi(Y_i-u(\mathbf{X}_i),\widehat{\theta}(\mathbf X_i))\mathbf{1}_{Y_i\geq u(\mathbf{X}_i)}$ is maximal. To shorten the notation, let $\varphi(Y_i,\theta)=\phi(Y_i-u(\mathbf{X}_i),\theta)\mathbf{1}_{Y_i\geq u(\mathbf{X}_i)}.$

A set of rules $(R_j)_{j\in J}$ is a set of maps such that $R_j(\mathbf{x})=1$ or 0 depending on whether some conditions are satisfied by $\mathbf{x}$, with $R_j(\mathbf{x})R_{j'}(\mathbf{x})=0$ for $j\neq j'$ and $\sum_{j}R_j(\mathbf{x})=1$. {In case of regression trees, these partitioning rules have a particular structure, since they can be written, for quantitative covariates (the case of $\mathbf{x}$ containing qualitative variables is described in Remark \ref{r_qual} below), as $R_j(\mathbf{x})=\mathbf{1}_{\mathbf{x}_1\leq \mathbf{x}<\mathbf{x}_2}$ for some $\mathbf{x}_1\in \mathbb{R}^d$ and $\mathbf{x}_2\in \mathbb{R}^d,$ with comparison symbols to be understood as component-wise comparisons. In other terms, if $d=1,$ rules can be identified as partitioning segments, if $d=2$ they are rectangles (hyper-rectangles in the general case).} The determination of these rules from one step to another can be represented as a binary tree, since each rule $R_j$ at step $k$ generates two rules $R_{j1}$ and $R_{j2}$ (with $R_{j1}(\mathbf{x})+R_{j2}(\mathbf{x})=0$ if $R_j(\mathbf{x})=0$) at step $k+1.$ The algorithm can be summarized as follows:

\noindent
\textbf{Step 1:} $R_1(\mathbf{x})=1$ for all $\mathbf{x},$ and $n_1=1$ (corresponds to the root of the tree).

\noindent
\textbf{Step k+1:} Let $(R_1,...R_{n_k})$ denote the rules obtained at step $k.$ For $j=1,\ldots,n_k,$
\begin{itemize}
  \item if all observations such that $R_j(\mathbf{X}_i)=1$ have the same characteristics,  
  then keep rule $j$ as it is no longer possible to segment the population;
  \item else, rule $R_j$ is replaced by two new rules $R_{j1}$ and $R_{j2}$ determined in the following way: for each component $X^{(\ell)}$ of $\mathbf{X}=(X^{(1)},\ldots,X^{(d)})$, define the best threshold $x^{(\ell)}_{j\star}$ to split the data, such that $x^{(\ell)}_{j\star} =  \arg\max_{x^{(\ell)}} \Phi(R_j,x^{(\ell)}),$ with
  \begin{eqnarray*}
    \Phi(R_j,x^{(\ell)}) & = & \sum_{i=1}^n \varphi(Y_i,\theta_{\ell-}(\mathbf{X}_i,R_j)) \mathbf{1}_{X_{i}^{(\ell)}\leq x^{(\ell)}} R_j(\mathbf{x})\\
    & + & \sum_{i=1}^n \varphi(Y_i,\theta_{\ell+}(\mathbf{X}_i,R_j)) \mathbf{1}_{X_{i}^{(\ell)} > x^{(\ell)}}  R_j(\mathbf{x}),
  \end{eqnarray*}
   where
   \begin{eqnarray*}   
   \widehat{\theta}(R_j)&=&\arg \max_{\theta\in \Theta}\sum_{i=1}^n \varphi(Y_i,\theta(\mathbf{X}_i))R_j(\mathbf{X}_i), \\
   \theta_{\ell-}(x,R_j) &=& \arg \max_{\theta\in \Theta} \sum_{i=1}^n \varphi(Y_i,\theta(\mathbf{X}_i))\mathbf{1}_{X_{i}^{(\ell)}\leq x}R_j(\mathbf{X}_i), \\
   \theta_{\ell+}(x,R_j) &=& \arg \max_{\theta\in \Theta} \sum_{i=1}^n \varphi(Y_i,\theta(\mathbf{X}_i))\mathbf{1}_{X_{i}^{(\ell)} > x}R_j(\mathbf{X}_i).
   \end{eqnarray*}

  Then, select the best component index to consider: $\widehat{\ell} = \arg\max_{\ell} \Phi(R_j,x^{(\ell)}_{j\star})$.\\
  Define the two new rules $R_{j1}(\mathbf{x})=R_j(\mathbf{x})\mathbf{1}_{x^{(\widehat{\ell})} \leq x^{(\widehat{\ell})}_{j\star}},$ and $R_{j2}(\mathbf{x})=R_j(\mathbf{x})\mathbf{1}_{x^{(\widehat{\ell})} > x^{(\widehat{\ell})}_{j\star}}.$
  \item Let $n_{k+1}$ denote the new number of rules.
\end{itemize}

\noindent
\textbf{Stopping rule:} stop if $n_{k+1}=n_k.$

\vskip 0.3cm
This algorithm has a binary tree structure. The list of rules $(R_j)_{1\leq j \leq n_k}$ are identified with the leaves of the tree at step $k,$ and the number of leaves of the tree is increasing from step $k$ to step $k+1.$ The stopping rule can also be slightly modified to ensure that there is a minimal number of points of the original data in each leaf of the tree at each step.

\begin{remark}
\label{r_qual}
In this version of the CART algorithm, all covariates are continuous or $\{0,1\}-$valued. For qualitative variables with more than two modalities, they must be transformed into binary variables, or the algorithm must be slightly modified so that the splitting step of each $R_j$ should be done by finding the best partition into two groups on the values of the modalities that minimizes the loss function. This can be done by ordering the modalities with respect to the average value---or the median value---of the response for observations associated with this modality.
\end{remark}

\subsubsection{From the tree to the parameter estimation}
\label{sec5}

From a given set of rules $\mathcal{R}=(R_j)_{j=1,\ldots,s},$ let $\mathcal{T}_j=\{\mathbf{x}:R_j(\mathbf{x})=1\},$ the $j$th leaf of the corresponding tree. The estimator $\widehat{\theta}$ associated with a tree $\mathcal{T}=(\mathcal{T}_{\ell})_{\ell=1,...,K}$ (where $K$ is the total number of leaves) is obtained as 
$$\widehat{\theta}(\mathbf{x})=\sum_{\ell=1}^K \widehat{\theta}(R_j)R_j(\mathbf{x})=\sum_{\ell=1}^K \widehat{\theta}_{\ell}\mathbf{1}_{\mathbf{x}\in \mathcal{T}_{\ell}}.$$

The maximal tree is the $T_{\max}$ obtained once the previous algorithm stops. It corresponds to a trivial estimator of $m,$ since either the number of observations in a leaf is one, or all observations in this leaf have the same characteristics $\mathbf{x}.$ 

The pruning step, presented in the next section, consists in extracting from the maximal tree a subtree that achieves a compromise between simplicity and good fit.

\subsubsection{Selection of a subtree: pruning step}
\label{subsec:pruning}

For the pruning step, a standard way to proceed is to use a penalized approach to select the appropriate subtree, see \citep[][]{breiman, gey}. For a given tree $T_K$ with $K$ leaves $(\mathcal{T}_{\ell})_{\ell=1,...,K},$ associated with the corresponding estimator $\widehat{\theta},$ the performance of this tree is measured through the following criterion
\begin{equation}
\label{c_pen}
\frac{1}{k_n}\sum_{\ell=1}^K \sum_{i=1}^n \varphi(Y_i-u,\widehat{\theta}(\mathbf{X}_i))\mathbf{1}_{\mathbf X_i \in \mathcal T_\ell}-\lambda K. 
\end{equation}

For a given level of penalty $\lambda,$ the selected tree is the one that maximizes criterion (\ref{c_pen}), achieving a compromise between good fit and simplicity. To determine this optimal tree, it is not necessary to compute all the subtrees from the maximal tree. It suffices to determine, for all $K\geq 0,$ the subtree $T_K$ which maximizes the criterion (\ref{c_pen}) among all subtrees with $K$ leaves, and then to determine the final tree among a list of $K_{\max}$ trees (where $K_{\max}$ is the number of leaves of the maximal tree). The trees $T_K$ are easy to determine, since $T_{K}$ is obtained by removing one leaf to $T_{K+1}$, see p.284--290 in \citep[][]{breiman}.

The penalization constant $\lambda$ can be chosen using a test sample or $k-$fold cross-validation. In the first case, data are split into two parts before making the tree grow (a training data of size $n$ and a test sample which is not used in computing the tree). In the second case, the dataset is randomly split into $k$ parts which successively act as a training or a test sample, see for example \citep{allen1974relationship,stone1974cross}.

\subsection{Comparison with competing approaches}
\label{subsec:compete}

Compared to competing approaches in extreme value regression, the advantage of the procedure is to introduce discontinuities in the regression function while parametric approaches suppose a form of linearity, e.g. \citep[][]{beirlant2003}. The more flexible non-parametric approaches, as \citep{beirlant2}, rely on smoothing techniques that require the covariates to be continuous. Chavez-Demoulin et al.  \citep{ChavezDemoulin2015} propose a semi-parametric framework to separate the continuous covariates from the discrete ones. Smoothing splines are used to estimate non-parametrically the continuous part, while the influence of discrete covariates is captured by a parametric function.

\section{Main results}\label{sec:results}
\label{sec_main}

In this section, we show that the GP regression tree procedure defined in Section \ref{sec_GPDtrees} is consistent. Notations and assumptions used throughout this section are listed in Section \ref{subsec:not}. We then state our first main results on the consistency of a fixed tree with $K$ leaves, by separating the stochastic part of the error (Section \ref{subsec:bounds:estimate}) from the misspecification part (Section \ref{subsec:misspec}) caused by the GP approximation. The consistency of the pruning methodology is studied in Section \ref{subsec:consistency:pruning}. 
 
\subsection{Notations}
\label{subsec:not}

Let us recall that the PoT approach consists in considering observations such that $Y_i\geq u(\mathbf{X}_i).$ Below, we will restrain ourselves to the case where $u(\mathbf{x})= u.$ Our results easily extend to the case where $u(\mathbf{x})=\sum_{j=1}^m u_j \mathbf{1}_{\mathbf{x}\in \mathcal{X}_j},$ where $(\mathcal{X}_j)_{1\leq j \leq m}$ are subsets of the space of covariates. Another possible extension would be to assume that $u(\mathbf{x})=f(\beta,\mathbf{x})$ for some parameter $\beta$ and $f$ a known function. Nevertheless, a choice of such a particular threshold function seems hard to justify. Hence, we restrain ourselves to the simplest case.

Moreover, the result we provide holds uniformly for $u\in [u_{\min},u_{\max}]$ to cover adaptive choice of this parameter. Conditions on $u_{\min}$ and $u_{\max}$ are given in Assumption \ref{a_u}.

\begin{assum}
\label{a_u}
If $n$ denote the number of observations, let $k_n$ be an intermediate sequence, that is $k_n \to \infty$ and $k_n/n \to 0$, as $n \to \infty$. Then, let $k_n/n$ denote the average proportion of $Y$ larger than $u_{\min},$ that is
$\mathbb{P}(Y\geq u_{\min})=k_n n^{-1}.$ 
Moreover, assume that
$$\mathbb{P}(Y\geq u_{\max})=\frac{u_0 k_n}{n},$$
for some constant $u_0\geq 1.$
\end{assum} 

Here, $k_n$ will denote the average number (up to some constant) of observations on which the model is fitted. It is hence related to the rate of convergence of the procedure. The following assumption introduces conditions on this rate $k_n$ and on the space of parameters.

\begin{assum}
\label{a_rate}
We assume the parameter space to be $\Theta=\mathcal{S}\times \Gamma$ where
\begin{itemize}
\item $\mathcal{S}=[\sigma_{\min},\sigma_n],$ with $\sigma_n=O(n^{a_1}),$ with $a_1 >0$,
\item $\Gamma$ is a compact set $[\gamma_{\min},\gamma_{\max}],$ with $\gamma_{\min}>0.$
\end{itemize}
Moreover, assume that $k_n=O(n^{a_2}),$ with $a_2 >0$, and that the number of leaves of the maximal tree $K_{\max}$ satisfies $K_{\max} \leq \kappa k_n,$ with $\kappa >0$. 

\end{assum}

Next, let us introduce some notations regarding the trees. Consider a tree $T(u)$ with $K$ leaves denoted $\mathcal T_\ell$, $\ell=1,\ldots, K$. Introducing the (normalized) contribution of the log-likelihood to the $\ell$th leaf, say
\[
L_n^\ell (\theta,u) = \frac{1}{k_n} \sum_{i=1}^n \phi(Y_i - u,\theta) \mathbf{1}_{Y_i >u} \mathbf{1}_{\mathbf X_i \in \mathcal T_\ell},
\]
let
$$\widehat{\theta}_\ell(u)=\arg\max_{\theta} L^\ell_n(\theta,u),$$ the estimated value of the parameter in the leaf $\mathcal T_\ell.$ This estimator is expected to be close to
$$\theta^*_{\ell}(u)=\arg \max_{\theta} L^{\ell}(\theta,u),$$
introducing $L^{\ell}(\theta,u)=k_n n^{-1}\esp[L^\ell_n(\theta,u)].$ We denote by $T^*(u|T)$ the tree with same leaves as $T,$ but with parameters $\theta^*_{\ell}(u).$ This quantity is not exactly our target: ideally, we would like to estimate
$$\theta_{0,\ell}(u)=(\sigma_0(\mathcal{T}_{\ell},u),\gamma_0(\mathcal{T}_{\ell})),$$
such that
\begin{equation}
\lim_{t \to \infty} \sup_{z>0} |\overline F_{t}(z  \mid \mathcal{T}_{\ell}) - \overline{H}_{\sigma_0(\mathcal{T}_{\ell},t),\gamma_0(\mathcal{T}_{\ell})}(z) | = 0,
\end{equation}
where
$\overline F_{t}(z  \mid \mathcal{T}_{\ell}) =\mathbb{P}(Y-t\geq z | \mathbf{X}\in \mathcal{T}_{\ell},Y\geq t).$ We denote $T_0(u|T)$ the tree with same leaves as $T$ but with parameters $\theta_{0,\ell}(u).$

If $\theta=(\theta_\ell)_{\ell=1,\ldots,K}$ denotes the set of parameters of a tree with $K$ leaves $(\mathcal T_\ell)_{\ell=1,\ldots,K}$, we will denote $\theta(\mathbf x)$ the function defined by
\[
\theta(\mathbf x) = \sum_{\ell = 1}^K \theta_\ell \mathbf 1_{\mathbf x \in \mathcal T_\ell} .
\]
We will first focus on the difference $T(u)$ and $T^*(u|T)$ in Section \ref{subsec:bounds:estimate}, which is the stochastic part of the error. On the other hand, the difference between $T^*(u|T)$ and $T_0(u|T)$ (and ultimately the difference between $\widehat{\theta}(\mathbf{x})$ and $\theta_0(\mathbf{x})$) is studied in Section \ref{subsec:misspec} and can be understood as a misspecification term, caused by the fact that the excesses above the threshold are not exactly GP distributed.

 For $\ell = 1,\ldots, K$, let $\nabla_{\theta} L^\ell (\theta,u)$ denote the gradient of $L^\ell (\theta,u),$ denoting
\[
\nabla_{\theta} L^\ell (\theta,u) = \esp
\left(
\begin{array}{cc}
g_{\theta,\ell}(Y_i - u) \\
h_{\theta,\ell}(Y_i - u) 
\end{array}
\mathbf{1}_{\mathbf X_i \in \mathcal T_\ell} \mathbf{1}_{Y_i > u}
\right)
\]
with, for $z>0$, 
\begin{eqnarray*} 
g_{\theta}(z) &=& \partial_\sigma \phi(z,\theta) = \left(-\frac1 \sigma + \left(1 + \frac 1 \gamma \right)\frac{\gamma z}{\sigma^2(1+ \frac{\gamma z}{\sigma})}\right), \\
h_{\theta}(z) &=& \partial_\gamma \phi(z,\theta) = \left(-\frac{1}{\gamma^2} \log\left(1+\frac{\gamma z}{\sigma}\right) + \left(1+\frac 1 \gamma \right)\frac{ z}{\sigma+\gamma z}\right). 
\end{eqnarray*}
To handle the stochastic part, we shall add a few assumptions. We first need a domination condition on the class of the derivatives of the functions $y\rightarrow \phi(y-u,\theta).$ These derivatives are uniformly bounded by
$$\Phi(y)=C(1+\log (1+wy)),$$
where $C$ is a constant (not depending on $n$), and $w=\gamma_{\max}/\sigma_{\min}.$

\begin{assum}
\label{a_phi} 
Assume that, for some $\rho_0>0,$
$$m_{\rho_0}=\mathbb{E}\left[\exp(\rho_0\Phi(Y))\right]<\infty.$$
\end{assum}

In fact, this assumption is automatically satisfied if Assumption \ref{a_rate} holds: since $\gamma(\mathbf{x})\geq \gamma_{\min}>0,$ $\esp[|Y|^{1/\gamma-\varepsilon}]<\infty,$ for any $\varepsilon>0.$

Additionally, we need some regularity assumptions on the criterion $L^{\ell}.$

\begin{assum}
\label{a_moche}
Let
$$M^\ell_{\theta_1,\theta_2,\theta_3,\theta_4}(u)=\mathbb{E}\left[\left(\begin{array}{cc} \partial_\sigma g_{\theta_1}(Y-u) & \partial_{\gamma} g_{\theta_2}(Y-u) \\ 
\partial_\sigma h_{\theta_3}(Y-u) & \partial_{\gamma} h_{\theta_4}(Y-u)\end{array}\right)\mathbf{1}_{Y\geq u}\mid \mathbf X \in \mathcal T_{\ell}\right].$$
Assume that there exists a constant $\mathfrak C_1>0$ such that 
$$\inf_{a,b\in \mathbb{R}}\inf_{\theta_1,\theta_2,\theta_3,\theta_4\in \Theta}\inf_{u \in[u_{\min}, u_{\max}]}\inf_{\ell=,\ldots,K} \left|M^\ell_{\theta_1,\theta_2,\theta_3,\theta_4}(u)\left(\begin{array}{c} a \\ b \end{array}\right)\right|\geq \mathfrak C_1\max(|a|,|b|).$$
\end{assum}

The condition on the infimum can be relaxed: Assumption \ref{a_moche} comes naturally in using a Taylor expansion. Hence, the infimum with respect of $\theta_1,\ldots,\theta_4$ can be restricted to $\theta_2$ to $\theta_3$ belonging to a small neighborhood of $\theta_1$ (and not to the whole set $\Theta$).

\subsection{Deviation bounds for our estimator}\label{subsec:bounds:estimate}

 In this section, we study the consistency of a fitted tree $T(u)$, a  subtree of the maximal tree $T_{\max}(u),$  with $K$ leaves $(\mathcal T_\ell)_{\ell=1,\ldots, K}$,. We compare this fitted tree to $T^*(u|T),$ which is the tree based on the same subdivision, but where, in each leaf $\ell,$ the parameter is $\theta^*_{\ell}(u)$ (instead of $\widehat{\theta}_{\ell}(u)$ in $T(u)$).
 
The first step is to define a distance between trees. Let us define $\|(a,b)\|_{\infty}=\max(|a|,|b|),$ and for two trees $T$ and $S$,
$$\|T-S\|_2=\left(\int \|T(\mathbf{x})-S(\mathbf{x})\|_{\infty}^2 \mathrm{d}\mathbb{P}(\mathbf{x})\right)^{1/2}.$$ The main result of this section is a deviation bound for $\|T(u)-T^*(u|T)\|_2,$ which is Theorem \ref{thm:deviation:bounds} below.

\begin{theorem}
\label{thm:deviation:bounds}

Under Assumptions \ref{a_u} to \ref{a_moche}, and let $\beta>0$ such that $\beta a_2\geq 10/\rho_0$ (with $\rho_0$ defined in Assumption \ref{a_phi}) and for $t\geq { c_1} K (\log k_n) k_n^{-1},$ with $c_1>0$, 
\begin{eqnarray*}
\lefteqn{\mathbb{P}\left(\sup_{u_{\min}\leq u\leq u_{\max}}\| T(u)- T^*(u|T)\|_2^2\geq t\right) }\\
&\leq & 2\left(\exp\left(- \frac{{ \mathcal C_1}  k_n t}{K\beta^2 (\log k_n)^{2}} \right) + \exp\left(-\frac{{ \mathcal C_2} k_n t^{1/2}}{K^{1/2}\beta \log k_n} \right)\right) +\frac{{ \mathcal C_3}K}{k_n^{5/2} t^{3/2}},
\end{eqnarray*}
where ${ \mathcal C_1}$, ${ \mathcal C_2}$ and ${ \mathcal C_3}$ are positive constants. 
\end{theorem}

The proof of Theorem \ref{thm:deviation:bounds} is postponed to the appendix section (Section \ref{sec1}). The exponential terms on the right-hand side come from concentration inequalities proved by Einmahl and Mason \citep{einmahl2005uniform}, while the polynomially decreasing term is related to the fact that the log-likelihood is an unbounded quantity, but that can still controlled when considering its expectation.

As a by-product, we obtain the following Corollary \ref{cor_exp} (by integration of the bound of Theorem \ref{thm:deviation:bounds}).

\begin{corollary}
\label{cor_exp}
$$\mathbb{E}\left[{ \sup_{u_{\min}\leq u\leq u_{\max}}\|T(u)- T^*(u|T)\|_2^2}\right]\leq { \mathcal C_4} \frac{K\beta ^2(\log k_n)^2}{k_n}.$$
\end{corollary}

From Corollary \ref{cor_exp}, one can see that the $L^2-$norm of the stochastic part of the error, $\mathbb{E}\left[\sup_{u_{\min}\leq u\leq u_{\max}}\| T(u)- T^*(u\mid T)\|_2^2\right]^{1/2},$ is proportional to $K^{1/2},$ and, as expected, increases with the complexity of the tree. On the other hand, the error decreases almost at rate $k_n^{1/2}$ (up to some logarithmic factor), which is the convergence rate of standard estimators used to estimate the tail parameter in absence of covariates.

The proof is again postponed to the appendix (Section \ref{sec2}).

\subsection{Misspecification bias}
\label{subsec:misspec}

For $\mathbf X = \mathbf x$, the ultimate goal is to estimate the tail index parameter $\theta_0(\mathbf x)=(\sigma_{0u}(\mathbf x),\gamma_0(\mathbf x))$, introduced in \eqref{eqplus2}, by maximization of the GP likelihood.  The difference between $\theta_0(\mathbf x)$ and $\theta^*(\mathbf x)$ can be understood as a misspecification term due to the fact that the observations above the threshold are not exactly distributed according to a GP distribution. This bias term can be controlled under second order conditions which are standard in Extreme Value Analysis.

Indeed, recall that assuming that the underlying distribution $\overline{F}(\cdot|\mathbf x)$ satisfies Condition \eqref{eq:rv2} guarantees that asymptotically the associate excesses above the threshold $u$ are GP distributed. For finite samples, the excesses are thus not exactly GP distributed which introduces some bias term. In order to control this bias term, a second-order condition is needed, that is a condition to control the rate of convergence in Condition \eqref{eq:rv2}. There exist numerous ways to express this second-order condition. Here, we consider the same condition as Condition C.6 in \citep{beirlant2}. First, Condition \eqref{eq:rv2} can be translated into 
\begin{equation}
\overline{F} (y\mid \mathbf x) = y^{-1/\gamma_0(\mathbf x)} \eta(y \mid \mathbf x)\, , \forall y>0,
\end{equation}
where $\eta$ is a slow-varying function, that is $\eta(ty\mid \mathbf x)/\eta(t\mid \mathbf x) \rightarrow 1$ as $t \to \infty$, for all $y>0$. 
 
\begin{assum}\label{assum:rv:second} Assume that for all $\mathbf x$, there exist a constant $c$ and a function $\psi$ such that
\begin{equation*}
\eta(ty \mid \mathbf x)/\eta(t\mid \mathbf x) = 1 + c\psi(t)\int_1^t v^{\rho-1} \mathd v + o(\psi(t ))
\end{equation*}
as $t \to \infty$ for each $y>0$ with $\psi(t)>0$ and $\psi(t) \to 0$ as $t \to \infty$ and $\rho\leq 0$. 
\end{assum}

Let us note that we could also consider the case of $c,$ $\psi$ and $\rho$ depending on $\mathbf{x},$ and then assume some uniform bound over $x$ of these quantities. We chose this more restrictive formulation to simplify the notations.

The next result guarantees that the bias term tends to 0 as $u \to \infty$. 
\begin{proposition}\label{prp:bias}
There exists a constant $c$ and a function $\psi $ such that $\psi(u)>0$ and $\psi(u) \to 0$ as $u \to \infty$, and such that, for $\mathbf X = \mathbf x$,
\begin{eqnarray*}
\|\theta_{0}(\mathbf x) - \theta^*(\mathbf x)\|_\infty
&\leq& \mathfrak C_2 (u)\frac{k_n}{n}  \left( 1 + c\gamma_{\max} \psi(u) + o(\psi(u)) \right) \,,
\end{eqnarray*}
where $\mathfrak C_2 (u)$ is a constant depending on $u$, $\gamma_{\min}$ and $\gamma_{\max}$.
\end{proposition}

\subsection{Consistency of the pruning step}\label{subsec:consistency:pruning}

The previous results cover the case of a tree with fixed number of leaves $K.$ In practice, the question is to select the proper subtree of $T_{\max}(u)$, the maximal tree obtained once the previous step of the CART procedure has stopped, with some ``optimal'' number of leaves, which is the objective of the pruning step described in Section \ref{subsec:pruning}.

As seen in Corollary \ref{cor_exp}, the stochastic part of the error put to the square increases proportionally to $K.$ This is closely related to the natural inflation of the log-likelihood (which is locally quadratic) when the number of leaves increases, justifying a penalty proportional to $K,$ as in \citep{breiman,gey}. The aim of Theorem \ref{thselection} is to corroborate this choice.

First of all, for a decomposition $(\mathcal{T}_{\ell}^K)_{\ell=1,...,K}$ of $K$ leaves, let us define $T_K(u)$ the tree with parameters $\widehat{\theta}^{K}_{\ell}(u)$ estimated with the CART procedure, $T^*_K(u)$ the tree with parameters
$$\theta^{*K}_{\ell}(u) =\arg \max_{\theta\in \Theta} \esp\left[\phi(Y - u,\theta) \mathbf{1}_{Y >u} \mathbf{1}_{\mathbf X_i \in \mathcal T_\ell^K}\right],$$
and $\mathbf{x}\rightarrow \theta^{*K}(\mathbf{x})=\sum_{\ell = 1}^K\theta^{*K}_{\ell}(u) \mathbf{1}_{\mathbf x \in \mathcal T_\ell^K}$ the corresponding regression function.
Moreover, let 
$$K_0(u)=\arg \max_{K=1,...,K_{\max}} \esp\left[\phi(Y - u,\theta^{*K}(\mathbf{X})) \mathbf{1}_{Y >u} \right].$$
In words, $T^{*}(u) = T^{*}_{K_0(u)}(u)$ is the subtree of $T_{\max}(u)$ that achieves the closest proximity to $\mathbf{x}\rightarrow \theta^{*K}(\mathbf{x})$ in the sense that it maximizes the expectation of the (pseudo)-log-likelihood.

Second of all, we denote, as explained in \eqref{c_pen}, the selected number of leaves
\[
\widehat{K}(u) = \arg\max_{K=1,\ldots, K_{\max}} \left\{\frac{1}{k_n}\sum_{\ell=1}^K\sum_{i=1}^n \phi(Y_i-u, \widehat{\theta}^{K}(\mathbf X_i)) \mathbf 1_{Y_i>u} \mathbf 1_{\mathbf X_i \in \mathcal T_\ell} - \alpha K  \right\} \, ,
\]
and $\widehat{T}(u)=T_{\widehat K(u)}(u)$ the corresponding selected tree. 

Define the log-likelihood $L_n(T_K,u)$ associated with a tree $T_K(u)$ with $K$ leaves $(\mathcal T_{\ell}^K)_{\ell = 1,\ldots, K}$ with parameters $\widehat \theta^K(u)=\left(\widehat\theta^K_{\ell}(u)\right)_{\ell=1,\ldots,K}$
\[
L_n(T_K,u) = \sum_{\ell = 1}^K L_n^{\ell}(\widehat \theta^K_\ell,u) \, .
\] 
Then $L(T_K,u) = \esp[L_n(T_K,u)]$. Finally, for two trees $T$ and $S$, $\Delta L_n(T, S) = L_n(T,u) - L_n(S,u)$ and similarly, $\Delta L(T, S) = L(T,u) - L(S,u)$.  

The following Theorem \ref{thselection} shows that the pruning methodology selects a tree $\widehat T(u)$ which approximately achieves the same rate as $T_{K_0}(u),$ even if $K_0(u)$ is unknown, provided that the penalty constant $\lambda$ belongs to some reasonable interval.

\begin{theorem}
\label{thselection}
Let $\mathfrak D = \inf_u \inf_{K < K_0(u)}  \Delta L(T^*(u),T^*_K(u))$ and suppose that there exists a constant $c_2>0$ sauch that the penalization constant $\lambda$ satisfies 
\[
 c_2  \{\log k_n\}^{1/2} k_n^{-1/2} \leq \lambda
\leq(\mathfrak{D} - 2{c_2 \{\log(k_n)\}^{1/2} k_n^{-1/2})k_n^{-1}},
\]
then, for all $u\in[u_{\min},u_{\max}]$,
$$\mathbb{E}\left[\|\widehat{T}(u)-T^*(u)\|^2_2\right]\leq \frac{\mathcal{C}_5  K_0(u) (\log k_n)^2 }{k_n},$$
where $\mathcal{C}_{5}$ is a constant depending on $T^*(u).$
\end{theorem}

The proof is given in Section \ref{sec3}.

\section{Simulation study and real data analysis}
\label{sec_simu}

This section is devoted to the illustration of the GP regression procedure on simulated data (Section \ref{subsec:simus}) and on a real dataset (Section \ref{subseb:realdata}).

\subsection{Simulations}
\label{subsec:simus}

In this section, we assess the performance of the GP regression procedure on simulated data and compare it with the competing approach proposed by \citep{ChavezDemoulin2015}. We first describe the simulation framework and then discuss the experiments results.

We consider the following regression framework: $X$ is a one-dimensional variable uniformly distributed on $[0,1]$, and the response variable $Y$, conditionally on $X=x$, is distributed according to a Burr distribution of parameters $(\sigma, \gamma_0(x))$ which survival function is given by  
\begin{equation*}
    \label{eq_burr}
    \overline{F}(y \mid  x) =\frac{1}{1+ \left(y/\sigma\right)^{1/\gamma_0(x)}} \, ,
\end{equation*}
with $\sigma>0$ and $\gamma_0(x)$ for all $x$. Note that $\overline{F}(\cdot \mid x)$ satisfies the property \eqref{eq:rv2}.

We consider two cases: (i) $\gamma_0(x)$ as a step-wise function and (ii)  $\gamma_0(x)$ as smooth function. In both cases, the scale parameter $\sigma$ was fixed equal to 1. 
\paragraph{(i) step-wise function:} In this case, the function $\gamma_0$ is taken as 
\begin{equation*}
\gamma_0(x) = \begin{cases}
0.5 & \quad \text{if $0 \leq x <0.25$}\\
1 & \quad \text{if $0.25 \leq x <0.75$}\\
1.5  & \quad \text{if $0.75 \leq x \leq 1$}. 
\end{cases}
\end{equation*}
\paragraph{(ii) smooth function:} In this case, the function $\gamma_0$ is taken as, for $x \in [0,1]$, 
\begin{equation*}  
\label{eq_simu_model}
  \gamma_0(x) = 1 + \frac{\tanh(10(x-1/4))}{4} + \frac{\tanh(10(x-3/4))}{4} \, . 
\end{equation*} 

We simulate 1~000 replications for different sizes of the observation sample ($n=$1000, 2500, 5000, 10~000 and 25~000) according to the described framework for both cases (i) and (ii). For each sample, we consider the excesses above the 0.90-empirical quantile, which corresponds to $k_n=$100, 250, 500, 1~000 and 2~500.  For each simulated sample, we compute the regression tree procedure (CART), and the method based on generalized additive model (GAM) proposed by \citep{ChavezDemoulin2015}. Next we compute $\int_0^1(\hat{\gamma}(x)-\gamma_0(x))^2dx$ for each estimator. The empirical mean squared error is then obtained by averaging these errors over the 1~000 replications. Results are shown in Table \ref{tab:mse}. The boxplots of the empirical quadratic errors are shown in the supplementary material (Section A).

\begin{table}[!ht]
\caption{Empirical mean squared errors for the GP regression tree procedure (GP CART), and the GAM model for different sample sizes for a) the step-wise case and  b) the smooth case. }
\begin{center}
\begin{tabular}{c}
\begin{tabular}{c|ccccc}
$k_n$ & 100 & 250 & 500 & 1~000 & 2~500\\
\hline
GP CART & 0.290 & 0.129 & 0.107 & 0.080 & 0.050\\
GAM & 0.313 & 0.196 & 0.122 & 0.081 & 0.048
\end{tabular} \\
a) \\
\begin{tabular}{c|ccccc}
$k_n$ & 100 & 250 & 500 & 1~000 & 2~500\\
\hline
GP CART & 0.227 & 0.108 & 0.079 & 0.059 & 0.043\\
GAM & 0.233 & 0.144 & 0.068 & 0.034 & 0.016
\end{tabular}\\
 b) 
\end{tabular}
\end{center}
\label{tab:mse}
\end{table}

Let us note that the GAM approach is not designed to capture non-smooth functions like in the step-wise case. Nevertheless, we see that this technique manages to fit relatively correctly even in this case when the sample size is large. For $k_n=1~000$ and $2~500,$ the results of the GAM approach are similar or even slightly better than the regression tree method. On the other hand, we observe that regression trees lead to a better fit for small sample sizes, even in the smooth case where it is not designed to take into account the regularity of $\gamma_0(x).$

\subsection{Prediction of the  cost of flooding events in France}
\label{subseb:realdata}

 In order to improve the knowledge and the management of natural catastrophes, the French Federation of Insurance (FFA) is interested in the prediction of the cost of such events, especially of the most severe ones, shortly after their occurrence. These catastrophic events present some heterogeneity in their intensity depending on their characteristics, such as the affected meteorological region or the number of individual houses in flood risk area. The prediction of their cost thus becomes a challenging task. In this section, we illustrate how the GP regression tree procedure can be used to gain further insight in this heterogeneity. The ability of the procedure to design classes of events that are more homogeneous (in view of analyzing the tail of their distribution) is an appealing property in view of operation applications in insurance.

The database we consider was obtained through a partnership with the FFA, in particular with one of its dedicated technical body, the association of French insurance undertaking for natural risk knowledge and reduction (Mission Risques Naturels, MRN). It consists of all 3~100 flooding events that have been granted the status of natural catastrophe in France from 1999 to 2019 (let us note that the status "natural catastrophe" is a French specificity, with some legal consequences when an event receives this label, see \citep{charpentier_insurance_2021,MRN}).This database is fed by 13 contributors including the major French insurance companies, allowing this database to cover 70\% of French non-life insurance market. The database gathers information regarding each flooding event (its cost, the meteorological region, the season, the number of affected hydrological regions, the number of individual houses and the number of professional business premises in flood-risk area). Note that, since the purpose of this database is  the fast prediction of the cost of a flooding event (as soon as possible after its occurrence), the variables that are registered correspond to quantities that are available before the event, or soon after it.

The variable of interest, the total cost of a flooding event, is highly volatile. Indeed, it ranges between $0$ and $394~376~000$ euros with an empirical variance equal to $1.77e+14$. Figure \ref{fig:map:floods} shows the average of the costs of the 10\% most onerous flooding events within each meteorological region. This highlights the heterogeneity of the severity of the most severe events. Furthermore, the top ten most onerous events represent 43\% of the total cost of this database and the top hundred 80\%.

\begin{figure}[!ht]
\begin{center}
\includegraphics[width = 0.8\textwidth]{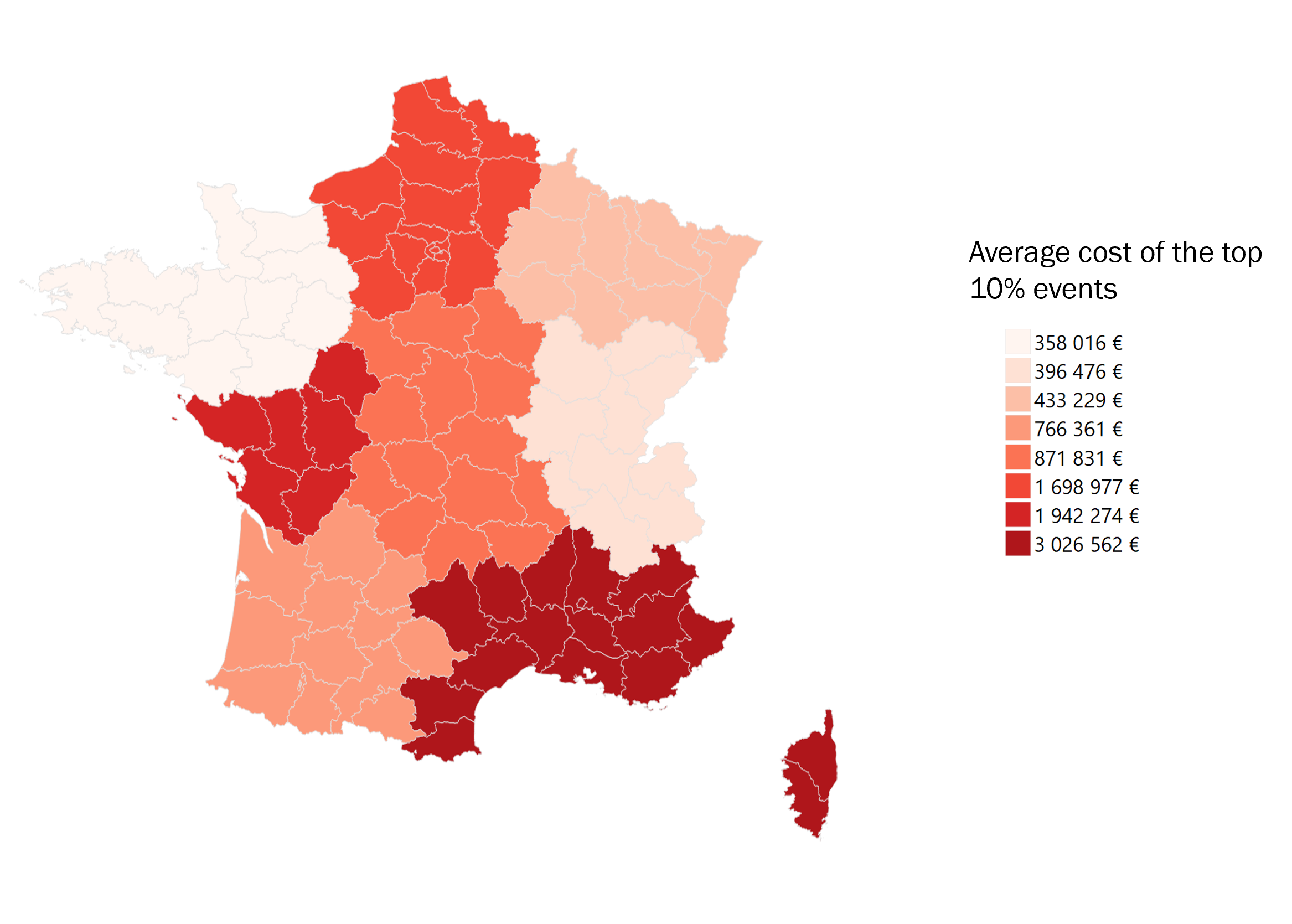}
\end{center}
\caption{Cartography of the cost of flooding events in France from 1999 to 2019. For each meteorological region, the average of the costs of the 10\% more onerous events is shown. The lighter red color suggesting a small cost while a darker color suggests a large cost.}
\label{fig:map:floods}
\end{figure}

 Now, let us recall that our goal is to understand the heterogeneity of the total cost of the most severe flooding events, that is of extreme flooding events. As explained in Section \ref{sec_evtreg}, the definition of extreme events consists in choosing a threshold $u$, which should be chosen as a bias-variance trade-off. We chose a value of $u=100~000$  based practical considerations and validated by sensitivity analyses (shown in the supplementary material, Section B). This yields 1~100 extreme events, that is for which the cost is larger than $u$.

The GP regression tree was performed on the database corresponding to the flooding events extracted from the original database for which the total cost is larger than $u$ (=100~000 euros). The variables of this database and their characteristics are summarized in Table \ref{tab:var:data}. Again, it can be noticed that the cost, the variable of interest, is highly volatile.

\begin{table}[!ht]
\caption{List of quantitative and categorical variables in the database and their characteristics. For the quantitative variables, Table a) shows the minimum, the first quartile, the median, the mean, the third quartile and the maximum, and for the categorical variables, Table b) the number of observations per category.}
\begin{center}
\begin{tabular}{c}
{\footnotesize
\begin{tabular}{m{3cm}|c|c|c|c|c|c}
 Variable & Min & 1st Q & Median & Mean & 3rd Q & Max\\
\hline
 Cost (in euros) & 100~005 & 183~901  & 390~761 & 4~949~576 & 1~339~936 & 394~376~166 \\
 \hline
Number of affected hydrological regions & 1 & 3 & 5 & 6.53 & 8 & 35\\
 \hline
 Number of individual houses  in flood risk area & 0 & 48~504 & 141~512 & 345~826 & 415~488 & 5~705~590 \\
  \hline
 Number of professional business premises in flood risk area & 0 & 17~525 & 54~921 & 168~950 & 185~772 & 2~431~039
 \end{tabular}
}\\
a) \\~\\
{\footnotesize
\begin{tabular}{c|c|c}
 Variable & Category & Number of observations\\
\hline
\multirow{8}{*}{Meteorological regions} 
& Center & 89\\
& North West& 111 \\
& North & 166\\
& North-East & 99 \\ 
& East & 135\\
& South & 281 \\
& West & 49\\
& South West & 158 \\
\hline
\multirow{4}{*}{Seasons} 
& Spring & 358\\
& Summer & 336\\
& Autumn & 251\\
& Winter & 143 \\
\hline
 \end{tabular}}
\\ 
b)
\end{tabular}
\end{center}
\label{tab:var:data}
\end{table}

The tree obtained from GP regression procedure is shown in Figure \ref{fig:tree:floods} (the quantile-quantile plots of the GP fit in each leaf are shown in the supplementary material, Section C). The tree is composed of 6 leaves, with three splits according to only 3 covariates: the number of individual houses, the number of professional business premises in flood-risk area and the number of affected meteorological regions. This seems reasonable since the first two covariates represent the exposure to floods, but also the population density of the affected area and the third one the extent of the flood. In each leaf, are given the shape and scale parameters. The worst case scenario corresponds to the leaf on the far right, with a shape parameter equal to 1 and containing 9\% of all flooding events. This leaf corresponds to events for which more than 9 meteorological regions are affected and more than 597~518 professional business premises are in flood-risk area.  The least severe case corresponds to the third leaf from the left, with a shape parameter equal to 0.24 and containing only 3\% of the events. Table \ref{tab:var:leaves} presents for each leaf the empirical median and mean of the costs and the theoretical median and mean of the corresponding GP distribution. Let us recall that for a GP distribution with a scale parameter $\sigma$ and a shape parameter $\gamma$, the theoretical median is given by $\sigma(2^\gamma-1)/\gamma$ and the empirical mean by $\sigma/(1-\gamma)$ for $\gamma<1$ and $\infty$ for $\gamma \geq 1$. First of all, for every leaf, the median is much smaller than the mean suggesting that we are indeed dealing with extreme events. Then, the empirical and theoretical medians are of the same order for each leaf while the empirical and theoretical (when it exits) means are only comparable for the leaves 3 and 5 for which the shape parameter  is significantly different from 1.

\begin{figure}[!ht]
\begin{center}
\includegraphics[width = \textwidth]{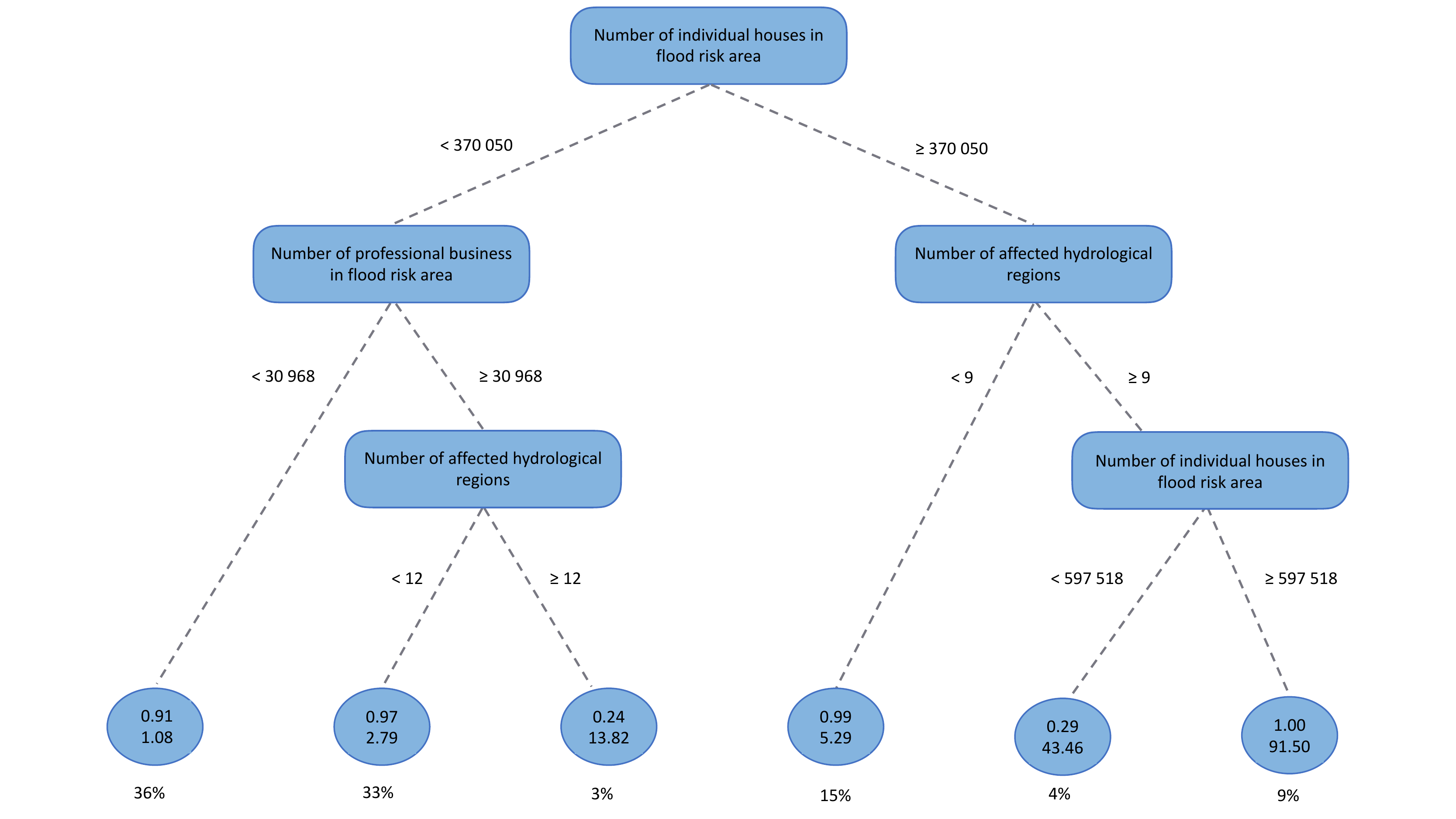}
\end{center}
\caption{GP regression tree obtained for flooding events. For each leaf, the value of the shape parameter $\gamma$ (first line) and the scale parameter $\sigma$ at $10^{-5}$ (second line) are given. Percentage of observations affected to each leaf is mentioned.}
\label{fig:tree:floods}
\end{figure}

\begin{table}
\begin{center}
{\footnotesize
\begin{tabular}{c|c|cc|cc}
Leaf & Shape parameter &  Empirical Median & Theoretical Median & Empirical Mean & Theoretical Mean\\
\hline
1 & 0.91 & 207~044 & 104~793 & 711~740 & 1~366~968 \\
2 & 0.97 & 364~513 & 276~879 & 1~325~493 & 13~168~585\\
3 & 0.24 & 900~945 & 1~045~203 & 1~929~512 & 1~938~357\\
4 & 0.99 & 578~437 & 529~377 & 3~868~125 & 807~158~756  \\
5 & 0.29 & 2~974~918 & 3~339~911 & 6~086~955 & 6~245~812\\
6 & 1.00 & 9~980~686 & 9~152~030 & 37~335~807 & $\infty$ \\
\end{tabular}}
\end{center}
\caption{Empirical median and mean, and theoretical median and mean for each leaf (in euros).}
\label{tab:var:leaves}
\end{table}

\section{Conclusion}

In this paper, we investigated the consistency of Generalized Pareto regression trees, applied to extreme value regression. The results that we derive are non-asymptotic, and allow to justify the consistency of the pruning methodology used to select a proper subtree. Let us note that the conditions under which our results hold are relatively weak, in the sense that they hold even if the tail index $\gamma$ is arbitrary close to zero (the special case $\gamma=0$ is excluded) or large. Moreover, no regularity assumptions on the target parameters is required, due to the flexibility of the regression tree procedure.

Through the simulation study and the real data analysis, we investigated the practical performances of the methodology. The regression tree approach can be applied in various situations, and still provides interpretability of the results. On the other hand, regression trees may be unstable, since quite sensitive to some changes on the data that have been used to fit them. Hence, this work is a first step into the direction of studying other relied methodologies, like random forests (see for example \citep{breiman}) in this field of extreme value regression.

\appendix

\section{Proofs}\label{sec:proofs}

In this Section, we present in details the proof of the results presented throughout the paper. Concentration inequalities required to obtain the results are presented in Section \ref{sec_cons}. These inequalities are used to obtain deviation bounds in Section \ref{sec_dev}, which are the key ingredients of the proof of Theorem \ref{thm:deviation:bounds} (Section \ref{sec1}), Corollary \ref{cor_exp} (Section \ref{sec2}), and Theorem \ref{thselection} (Section \ref{sec3}). Section \ref{sec_covering} shows some results on covering numbers that are required to control the complexity of some classes of functions considered in the proofs. Some technical lemmas are gathered in Section \ref{sec4}.

\subsection{Concentration inequalities}
\label{sec_cons}

The proofs of the main results are mostly based on concentration inequalities. The following inequality was proved initially by Talagrand \citep{talagrand1994sharper}, see also \citep[][]{einmahl2005uniform}. 

\begin{proposition} \label{talagrand} 
Let $(\mathbf V_i)_{1\leq i \leq n}$ denote i.i.d. replications of a random vector $\mathbf V,$ and let $(\varepsilon_i)_{1\leq i \leq n}$ denote a vector of i.i.d. Rademacher variables (that is, $\mathbb{P}(\varepsilon_i=-1)=\mathbb{P}(\varepsilon_i=1)=1/2)$ independent from $(\mathbf V_i)_{1\leq i \leq n}.$
Let $\mathfrak {F}$ be a pointwise measurable class of functions bounded by a finite constant $M_0.$ Then, for all $t,$
\begin{eqnarray*}
\mathbb{P}\left(\sup_{\varphi \in \mathfrak{F}}\left\|\sum_{i=1}^n \{\varphi(\mathbf V_i)-\mathbb \esp [\varphi(\mathbf V)]\}\right\|_{\infty}>A_1\left\{\mathbb \mathbb{E}\left[\sup_{\varphi\in \mathfrak{F}}\left\|\sum_{i=1}^n \varphi(\mathbf V_i)\varepsilon_i\right\|_{\infty}\right]+t\right\}\right) \\
\leq 2\left\{\exp\left(-\frac{A_2t^2}{nv_{\mathfrak{F}}}\right)+\exp\left(-\frac{A_2t}{M_0}\right)\right\},
\end{eqnarray*}
with $v_{\mathfrak{F}}=\sup_{\varphi \in \mathfrak{F}}\mathrm{Var}(\|\varphi(\mathbf V)\|_{\infty}),$ and where $A_1$ and $A_2$ are universal constants.
\end{proposition}

The difficulty in using Proposition \ref{talagrand} comes from the need to control the symmetrized quantity $\mathbb{E}\left[\sup_{\varphi\in \mathfrak{F}}\left\|\sum_{i=1}^n \varphi(\mathbf V_i)\varepsilon_i\right\|\right].$ Proposition \ref{emason} is due to Einmahl and Mason \citep{einmahl2005uniform} and allows this control via some assumptions on the considered class of functions $\mathfrak{F}$.

We first need to introduce some notations regarding covering numbers of a class of functions. More details can be found for example in Chapter 2.6 of \citep[][]{vandervaart}. Let us consider a class of functions $\mathfrak{F}$ with envelope $\Phi$ (which means that for (almost) all $v,$ $f\in \mathfrak{F},$ $|f(v)|\leq \Phi(v)$). Then, for any probability measure $\mathbb{Q},$ introduce $N(\varepsilon,\mathfrak{F},\mathbb{Q})$ the minimum number of $L^2(\mathbb{Q})$ balls of radius $\varepsilon$ to cover the class $\mathfrak{F}.$ Then, define $$\mathcal{N}_{\Phi}(\varepsilon,\mathfrak{F})=\sup_{\mathbb{Q}:\mathbb{Q}(\Phi^2)<\infty} N(\varepsilon(\mathbb{Q}(\Phi^2)^{1/2}),\mathfrak{F},\mathbb{Q}).$$

\begin{proposition} \label{emason}
Let $\mathfrak{F}$ be a point-wise measurable class of functions bounded by $M_0$ with envelope $\Phi$ such that, for some constants $A_3, \alpha\geq 1,$ and $0\leq \sqrt{v} \leq M_0,$ we have
\begin{enumerate}
\item[(i)] $\mathcal{N}_{\Phi}(\varepsilon,\mathfrak{F})\leq A_3 \varepsilon^{-\alpha},$ for $0<\varepsilon<1,$
\item[(ii)] $\sup_{\varphi \in \mathfrak{F}}\mathbb{E}\left[\varphi(\mathbf V)^2\right]\leq v,$
\item[(iii)] $M_0\leq \frac{1}{4\alpha^{1/2}}\sqrt{nv/\log(A_4M_0/\sqrt{v}) },$ with $A_4=\max(e,A_3^{1/\alpha}).$
\end{enumerate}
Then, for some absolute constant $A_5,$
$$\mathbb{E}\left[\sup_{\varphi\in \mathfrak{F}}\left\|\sum_{i=1}^n \varphi(\mathbf V_i)\varepsilon_i\right\|\right]\leq  A_5\sqrt{\alpha n v \log(A_4M_0/\sqrt{v})}.$$
\end{proposition}

\subsection{Deviation results}
\label{sec_dev}

We first introduce some notations that will be used throughout Sections \ref{sec_dev} to \ref{sec_covering}. In the following, $f_{\theta}$ is a function indexed by $\theta=(\sigma,\gamma)^{\tau}$ denoting either $\phi(\cdot,\theta)$ or $g_{\theta}=\partial_{\sigma} \phi(\cdot,\theta),$ or $h_{\theta}=\partial_{\gamma}\phi(\cdot,\theta).$ Let us note that the functions $y \mapsto g_\theta(y-u)$ and $y \mapsto h_\theta(y-u)$ are uniformly bounded (eventually up to some multiplication by a constant) by $\Phi(y)=\log(1+wy),$ where $w=\gamma_{\max}/\sigma_{\min}$ (see Assumption \ref{a_phi}). On the other hand, $y \mapsto \phi(y-u,\theta)$ is bounded by $\log \sigma_n+\Phi(y)=O(\log(k_n))+\Phi(y).$ We consider in the following a class of functions $\mathfrak{F}$ defined as
\begin{equation}\mathfrak F =\left\{y \mapsto f_{ \theta}(y-u)\mathbf{1}_{y\geq u}\mathbf{1}_{\mathbf{x}\in \mathcal{T}_\ell}, \; \theta \in \Theta,\; u\in [u_{\min};u_{\max}],\ell=1,...,K\right\}.
\label{eq_frakF}
\end{equation}

Next, recall that for $\ell=1,\ldots,K$
$$L_n^{\ell}(\theta,u)=\frac{1}{k_n}\sum_{i=1}^n \phi(Y_i-u,\theta)\mathbf{1}_{Y_i>u}\mathbf{1}_{\mathbf{X}_i\in \mathcal{T}_{\ell}},$$
is the (normalized) GP log-likelihood in the leaf $\ell$ of the tree $T(u)=(\mathcal T_\ell)_{\ell=1,\ldots,K}$. The key results behind Theorems \ref{thm:deviation:bounds} and \ref{thselection} relies on studying the deviation of the processes 
\begin{eqnarray}
\nonumber
\mathcal{W}_0^{\ell}(\theta,u) &=& L_n^{\ell}(\theta,u)-L^{\ell}(\theta,u), \\
\mathcal{W}_1^\ell (\theta,u) &=& \nabla_\theta L_n^\ell (\theta,u)-\nabla_\theta L^\ell (\theta,u), \nonumber
\end{eqnarray}
indexed by $\theta,\; u$ and $\ell.$

We study these deviations by decomposing
$\mathcal{W}_i^\ell (\theta,u),$ for $i=0,1,$ (which is a sum of i.i.d. observations) into two sums:
\begin{itemize}
\item the first one gathers observations smaller than some bound (more precisely, such that $\Phi(Y_i)\leq M_n$), which is considered in Theorem \ref{th1}. Since these observations are bounded (even if this bound in fact depends on $n$ and can tend to infinity when $n$ grows), we can apply a concentration inequality such as the one of Section \ref{sec_cons};
\item in the second one, we consider the observations larger than this bound, and control them through the fact that the function $\Phi$ is assumed to have a finite exponential moment (see Assumption \ref{a_phi}).
\end{itemize}

Corollary \ref{cor_thetal}, which provides deviation bounds for estimation errors in the leaves of the tree, is then a direct consequence.

\begin{theorem}
\label{th1} 
Let $M_n=\beta \log k_n,$ with $\beta >0$ and 
\[
\underline{\mathcal Z}(M_n) = \sup_{f\in \mathfrak{F}} \left| \frac{1}{k_n} \sum_{i=1}^n \left(f(Y_i) \mathbf{1}_{\Phi(Y_i)\leq M_n} - \esp \left[ f(Y_i) \mathbf{1}_{\Phi(Y_i) \leq M_n}\right]\right) \right|
\]
Then, under Assumptions \ref{a_u}, \ref{a_rate} and \ref{a_moche},
\begin{equation}\label{Zmoins}
\mathbb{P}\left(\underline{\mathcal Z}(M_n) \geq t\right) \leq 2\left(\exp\left(- \frac{{ C_1}  k_n t^2}{M_n^2} \right) + \exp\left(-\frac{{ C_2} k_n t}{M_n} \right)\right),
\end{equation}
for $t \geq {\mathfrak c_1} (\log k_n)^{1/2} k_n^{-1/2}$.
\end{theorem}

\begin{proof} Let us stress that $\sup_{f \in \mathfrak F} \|f(y)\mathbf{1}_{\Phi(y)\leq M_n}\|_{\infty} \leq M_n.$ From Proposition \ref{talagrand}, 
\begin{eqnarray}\label{Zint}
&&\mathbb{P} \left(\underline{\mathcal{Z}}(M_n) \geq A_1 \left\{\mathbb{E} \left[\sup_{f\in \mathfrak{F}}\frac{1}{k_n}\left|\sum_{i=1}^n f(Y_i)\mathbf{1}_{\Phi(Y_i)\leq M_n}\varepsilon_i\right|\right]  +t\right\}  \right) \\
&\leq& 2 \left( \exp\left(- \frac{A_2 k_n^2 t^2}{nv_{\mathfrak F}} \right) + \exp\left(-\frac{A_2 k_n t}{M_n} \right)\right) \, . \notag
\end{eqnarray}
From Lemma \ref{lemma_vphi}, $v_{\mathfrak{F}}\leq M_n^2k_nn^{-1},$ which shows that the first exponential term on the right-hand side of (\ref{Zint}) is smaller than
\begin{equation}
\label{t1}
\exp\left(- \frac{A_2 k_n t^2}{M_n^2} \right).
\end{equation}
We can now apply Proposition \ref{emason} (combined with Lemma \ref{lem:cov:numbers}) to this class of functions with $v=M_n^2k_nn^{-1}$ and $M_0=M_n.$
Hence, 
\[
\mathbb{E} \left[\sup_{f\in \mathfrak{F}}\frac{1}{k_n}\left|\sum_{i=1}^n f(Y_i)\mathbf{1}_{\Phi(Y_i)\leq M_n} \varepsilon_i\right|\right] \leq \frac{{ A_6}}{k_n}\sqrt{nv \mathfrak{s}_n}={ A_6} \frac{ \mathfrak{s}^{1/2}_n}{k_n^{1/2}} \; ,
\]
where ${ A'_6}>0$ and $\mathfrak{s}_n=\log(\sigma_n^{\alpha} K^{4(d+1)(d+2)}n/k_n)$ ($\alpha>0$ being defined in Lemma \ref{lem:cov:numbers}).
From Assumption \ref{a_rate}, we see that $\mathfrak{s}_n=O(\log (k_n))$ (let us recall that $K$ is necessarily less than $n).$
Whence, if ${ \mathfrak c_1}= 2A_1{A'_6}$, for $t\geq {\mathfrak c_1} \left\{\log\left(k_n\right)\right\}^{1/2} k_n^{-1/2}$, 
\[
\mathbb{P}\left(\underline{\mathcal Z}(M_n) \geq t\right) \leq \mathbb{P} \left(\underline{\mathcal Z}(M_n) \geq A_1 \left\{\mathbb{E} \left[\sup_{f\in \mathfrak{F}}\frac{1}{k_n}\left|\sum_{i=1}^n f(Y_i)\mathbf{1}_{\Phi(Y_i)\leq M_n}\varepsilon_i\right|\right] + \frac{t}{2A_1}\right\} \right) \; .
\]
Equation (\ref{Zmoins}) follows from (\ref{Zint}) and (\ref{t1}) with ${{C}_1}=A_2A_1^{-2}/4$ and ${ C_2}=A_2A_1^{-1}/2.$
\end{proof}

\begin{theorem}
\label{th2}
Define
\[
\overline{\mathcal Z}(M_n)=\sup_{f\in \mathfrak{F}} \left|\frac{1}{k_n} \sum_{i=1}^n \left( f(Y_i )\mathbf{1}_{\Phi(Y_i) > M_n}\right)- \esp \left[  f(Y_i)\mathbf{1}_{\Phi(Y_i) > M_n}\right] \right|.
\]
Then, under Assumptions \ref{a_u}, \ref{a_rate} and \ref{a_phi}, for $M_n=\beta \log k_n=\beta a_2 \log n$ and $\beta a_2 \geq 10/\rho_0,$ and $t\geq {\mathfrak c_2} k_n^{-1/2},$
\begin{equation}\label{res:thm2}
\mathbb{P}\left(\overline{\mathcal Z}(M_n) \geq t\right) \leq   \frac{{ C_{3}}}{k_n^{5/2} t^3}. 
\end{equation}
\end{theorem}

\begin{proof}

Let $\beta'=\beta a_2.$
$\overline{\mathcal Z}(M_n)$ is upper-bounded by 
\[
\frac{1}{k_n} \sum_{i=1}^n \left\{\Phi(Y_i) \mathbf{1}_{\Phi(Y_i) \geq M_n}\mathbf{1}_{Y_i\geq u_{\min}} + \mathbb{E} \left[ \Phi(Y) \mathbf{1}_{\Phi(Y) \geq M_n}\mathbf{1}_{Y\geq u_{\min}}\right]\right\} \, . 
\]

A bound for $E_{1,n}=\mathbb{E} \left[ \Phi(Y) \mathbf{1}_{\Phi(Y) \geq M_n}\mathbf{1}_{Y\geq u_{\min}}\right]$ is obtained from Lemma \ref{lemma_moments}, and $nE_{1,n}/k_{n}\leq \mathfrak{e}_1 k_n^{-1/2}$ if $\beta' \geq 2/\rho_0.$

Next, from Markov inequality,
\begin{eqnarray*}t^3\mathbb{P}\left(\frac{1}{k_n} \sum_{i=1}^n \Phi(Y_i) \mathbf{1}_{\Phi(Y_i) \geq M_n}\mathbf{1}_{Y_i\geq u_{\min}}\geq t\right)&\leq & \frac{n E_{3,n}}{k_n^3}+\frac{n(n-1)E_{2,n}E_{1,n}}{k_n^3}\\
&&+\frac{n(n-1)(n-2)E_{1,n}^3}{k_n^3}.
\end{eqnarray*}
From Lemma \ref{lemma_moments}, we get
\begin{eqnarray*}
 \frac{n E_{3,n}}{k_n^3} &\leq & \frac{\mathfrak{e}_3 n^{-(\rho_0\beta'/4-1/2)}}{k_n^{5/2}}, \\
 \frac{n(n-1)E_{2,n}E_{1,n}}{k_n^3} &\leq & \frac{\mathfrak{e}_2\mathfrak{e}_1n^{-(\rho_0\beta'/2-3/2)}}{k_n^{5/2}}, \\
 \frac{n(n-1)(n-2)E_{1,n}^3}{k_n^3} &\leq & \frac{\mathfrak{e}_1^3n^{-(\rho_0\beta'/4-5/2)}}{k_n^{5/2}}.
\end{eqnarray*} Each of these terms is bounded by $\max(\mathfrak{e}_3,\mathfrak{e_2}\mathfrak{e}_1,\mathfrak{e}_1^3)k_n^{-5/2}$ for $\beta' \geq 10/\rho_0.$
Thus, for $t\geq 2 \mathfrak{e}_1 k_n^{-1/2}$ and $\beta' \geq 10/\rho_0,$ 
\begin{eqnarray*}
\lefteqn{
\mathbb{P}\left(\overline{\mathcal Z}_n \geq t \right) }\\
&\leq & \mathbb{P} \left(\frac{1}{k_n} \sum_{i=1}^n \Phi(Y_i) \mathbf{1}_{\Phi(Y_i) \geq M_n}\mathbf{1}_{Y_i\geq u_{\min}} \geq \frac{t}{2}\right) + \mathbb{P}\left( \mathbb{E} \left[ \Phi(Y) \mathbf{1}_{\Phi(Y) \geq M_n}\mathbf{1}_{Y\geq u_{\min}}\right] \geq \frac{t}{2}\right)\\
&\leq & \frac{8\max(\mathfrak{e}_3,\mathfrak{e_2}\mathfrak{c}_1,\mathfrak{e}_1^3)}{t^3k_n^{5/2}}
\end{eqnarray*}
\end{proof}

We now apply these results to deduce deviation bounds on the estimators $\widehat{\theta}_{\ell}$ in the leaves of the tree.

\begin{corollary}
\label{cor_thetal}
Under the assumptions of Theorem \ref{th1} and \ref{th2} and Assumption \ref{a_moche}, for $t\geq \mathfrak c_3 (\log k_n)^{1/2}k_n^{-1/2},$
\begin{eqnarray*}
\mathbb{P}\left(\sup_{\substack{\ell=1,\ldots, K,\\ { u_{\min} \leq u \leq u_{\max}}}} \|\widehat{\theta}_\ell(u)-\theta^*_\ell(u)\|_{\infty}\geq t\right)\leq 2\left(\exp\left(- \frac{{ C_4}  k_n t^2}{\beta^2 (\log k_n)^{2}} \right) + \exp\left(-\frac{{ C_5} k_n t}{\beta \log k_n} \right)\right)+\frac{{C_6}}{k_n^{5/2} t^3}.
\end{eqnarray*}
\end{corollary}

\begin{proof} For $1 \leq \ell \leq K$ and ${ u_{\min} \leq u \leq u_{\max}}$, write $\theta=(s,\gamma)^{\tau}$ and $\theta^*_\ell(u)=(s^*_\ell(u),\gamma^*_\ell(u))^{\tau},$ and let $m_{u,\ell}(\theta) =\nabla_{\theta}L^\ell(\theta,u).$
 From a Taylor expansion,
$$m_{u,\ell}(\theta)=\mathbb{E}\left[\left(\begin{array}{cc} \partial_s g_{\tilde{s}_1,\gamma}(Y-u) & \partial_{\gamma} g_{s,\tilde{\gamma}_1}(Y-u) \\ 
\partial_s h_{\tilde{s}_2,\gamma}(Y-u) & \partial_{\gamma} h_{s,\tilde{\gamma}_2}(Y-u)\end{array}\right)\mathbf{1}_{\mathbf{X}\in \mathcal T_\ell}\mathbf{1}_{Y\geq u}\right](\theta-\theta^*_\ell(u))^\tau,$$
for some parameters $\tilde{\gamma}_j$ (resp. $\tilde{s}_j$) between $\gamma$ and $\gamma^*_\ell(u)$ (resp. $s$ and $s^*_\ell (u)$). From Assumption \ref{a_moche}, we get, for all $\ell=1,\ldots, K$, 
$$\frac{n}{k_n}\|m_{u,\ell}(\theta)\|_{\infty}\geq \mathfrak C_1\|\theta-\theta^*_{\ell}(u)\|_{\infty}.$$
Hence, for all $\ell=1,\ldots, K$,
$$\mathbb{P}\left(\|\widehat{\theta}_\ell(u)-\theta^*_\ell(u)\|_{\infty}\geq t\right)\leq \mathbb{P}\left(\frac{n}{k_n}\|m_{u,\ell}(\widehat{\theta})\|_{\infty}\geq \mathfrak C_1t\right).$$
Since for all $\ell=1,\ldots, K$, $\nabla_{\theta} L_n^\ell(\widehat{\theta})=0,$ $\mathcal{W}_1^\ell(\widehat{\theta}(u),u)=-\frac{n}{k_n}m_{u,\ell}(\widehat{\theta}).$
Hence,
$$\mathbb{P}\left(\sup_{\substack{\ell=1,\ldots, K, \\{ u_{\min} \leq u \leq u_{\max}}}} \|\widehat{\theta}_\ell(u)-\theta^*_l(u)\|_{\infty}\geq t\right)\leq \mathbb{P}\left(\sup_{\substack{\ell=1,\ldots,K,\\{ u_{\min} \leq u \leq u_{\max}}}}\|\mathcal{W}_1^\ell(\widehat{\theta}(u),u)\|_{\infty}\geq \mathfrak C_1t\right),$$
and the right-hand side is bounded by
$$\mathbb{P}\left(\overline{\mathcal Z}(M_n)\geq \frac{\mathfrak C_1t}{2}\right)+\mathbb{P}\left(\underline{\mathcal Z}(M_n)\geq \frac{\mathfrak C_1t}{2}\right).$$
The result follows from Theorem \ref{th1} and \ref{th2}.
\end{proof}

\subsection{Proof of Theorem \ref{thm:deviation:bounds}}
\label{sec1}

The proof of Theorem \ref{thm:deviation:bounds} then consists in gathering the results on the leaves obtained in Corollary \ref{cor_thetal}. Let ${ u_{\min} \leq u \leq u_{\max}}$, 
$$\|T(u)-T^*(u|T)\|_2^2\leq \sum_{\ell=1}^K \|\widehat{\theta}_\ell(u)-\theta^*_\ell(u)\|_{\infty}^2\leq K \sup_{\ell=1,...,K}\|\widehat{\theta}_\ell(u)-\theta^*_\ell(u)\|_{\infty}^2.$$
Hence
\begin{align*}\mathbb{P}\left(\sup_{u_{\min} \leq u \leq u_{\max}}\|T(u)-T^*(u|T)\|_2^2\geq t\right)   \leq \mathbb{P}\left(\sup_{\substack{\ell=1,\ldots,K,\\{ u_{\min} \leq u \leq u_{\max}}}}\|\widehat{\theta}_\ell(u)-\theta^*_\ell(u)\|_{\infty}\geq t^{1/2}K^{-1/2}\right).
\end{align*}
The results follows from Corollary \ref{cor_thetal}, and from the assumption on $K\leq K_{\max}=O(k_n^3).$

\subsection{Proof of Corollary \ref{cor_exp}}
\label{sec2}

Write
$$\mathbb{E}\left[{ \sup_{u_{\min} \leq u \leq u_{\max}}}\|T(u)-T^*(u|T)\|_2^2\right]=\int_0^{\infty} \mathbb{P}({ \sup_{u_{\min} \leq u \leq u_{\max}}} \|T(u)-T^*(u|T)\|_2^2\geq t)dt.$$
Let $t_n=c_1 K (\log k_n)k_n^{-1},$ then
$$\int_0^{\infty} \mathbb{P}({ \sup_{u_{\min} \leq u \leq u_{\max}}}\|T(u)-T^*(u|T)\|_2^2\geq t)dt\leq t_n+\int_{t_n}^{\infty} \mathbb{P}({ \sup_{u_{\min} \leq u \leq u_{\max}}}\|T(u)-T^*(u|T)\|_2^2\geq t)dt.$$ We now use Theorem \ref{thm:deviation:bounds} to bound the integral on the right-hand side. Since $\int_0^{\infty}\exp(-a t)dt=\frac{1}{a},$ $\int_0^{\infty} \exp(-a^{1/2} t^{1/2})dt=\frac 2 a,$ and { $\int_1^{\infty}t^{-3/2}dt=2,$} we get
\begin{eqnarray*}
\mathbb{E}\left[{ \sup_{u_{\min} \leq
 u \leq u_{\max}}}\|T(u)-T^*(u|T)\|_2^2\right]&\leq &{ t_n+ \frac{2K\beta^2  (\log k_n)^2}{\mathcal C_1 k_n }+\frac{4K \beta^2 (\log k_n)^2}{\mathcal C_2^2 k_n } + \frac{2 \mathcal C_3K}{k_n^{5/2} }}\\
& \leq & \frac{c_1 K \log k_n}{k_n}+ \frac{2 K  \beta^2 (\log k_n)^2}{\mathcal C_1 k_n }+\frac{4K \beta^2(\log k_n)^2}{\mathcal C_2^2 k_n } + \frac{2 \mathcal C_3K}{k_n^{5/2} }\\
& \leq & \frac{\mathcal C_4 K (\log k_n)^2}{k_n}. 
\end{eqnarray*}

\subsection{Proof of Proposition \ref{prp:bias}}
Let $\mathbf x$ fixed, then, 
\[
\|\theta^*(\mathbf x) - \theta_0(\mathbf x)\|_\infty = \|\sum_{\ell=1}^{K_{\max}}\left(\theta^*_\ell - \theta_{0}(\mathbf x)  \right) \mathbf 1_{\mathbf x \in \mathcal T_\ell}\|_\infty \leq \sum_{\ell=1}^{K_{\max}}\|\theta^*_\ell - \theta_{0}(\mathbf x) \|_\infty \mathbf 1_{\mathbf x \in \mathcal T_\ell} \, .
\]

Now, from Taylor expansion, for $\ell=1,\ldots, K$, conditionally on $\mathbf X \in \mathcal T_\ell$, 
\begin{eqnarray*}
\nabla_\theta L^\ell(\theta_{0}(\mathbf x),u) &= & \nabla_\theta L^\ell(\theta^*_\ell,u) +  \nabla^2_\theta L^\ell (\widetilde{\theta}_\ell)(\theta_{0}(\mathbf x) - \theta^*_\ell)^\tau \\
&=& 0 +  \mathbb{E}\left[\left(\begin{array}{cc} \partial_\sigma g_{\tilde{\sigma}_1,\gamma}(Y-u) & \partial_{\gamma} g_{\sigma,\tilde{\gamma}_1}(Y-u) \\ 
\partial_\sigma h_{\tilde{\sigma}_2,\gamma}(Y-u) & \partial_{\gamma} h_{\sigma,\tilde{\gamma}_2}(Y-u)\end{array}\right)\mathbf{1}_{Y\geq u} \mid \mathbf X \in \mathcal T_\ell \right](\theta_{0}(\mathbf x) - \theta^*_\ell)^\tau \\
\end{eqnarray*}
for some parameters $\tilde \gamma_j$ (resp. $\tilde \sigma_j$)  between $\gamma_{0}(\mathbf x)$ and $\gamma^*_\ell$ (resp. $\sigma_{0}(\mathbf x)$ and $\sigma^*_\ell$).

Thus, under Assumption \ref{a_moche},  
\begin{eqnarray*}
\|\theta_{0}(\mathbf x) - \theta^*_\ell\|_\infty &\leq&  \frac{1}{\mathfrak C_1} \|\nabla_\theta L^\ell(\theta_{0}(\mathbf x),u) \|_\infty \\
&\leq&  \frac{1}{\mathfrak C_1}\frac{k_n}{n} \max\left(| \esp\left[g_{\theta_{0}(\mathbf x)}(Z) \mid  \mathbf X \in \mathcal T_\ell\right]|,\esp\left[h_{\theta_{0}(\mathbf x)}(Z) \mid  \mathbf X \in \mathcal T_\ell\right] \right) \,,
\end{eqnarray*}
where $Z$ is a random variable distributed according to the distribution $F_u$ defined in Section \ref{sec_evtreg} with $\sigma_0(\mathbf x) =u\gamma_0(\mathbf x)$ and with 
\begin{eqnarray*}
\esp\left[g_{\theta_{0}(\mathbf x)}(Z) \mid  \mathbf X \in \mathcal T_\ell\right] &=& -\frac{1}{u \gamma_{0}(\mathbf x)} + \frac{1}{u^2\gamma_{0}(\mathbf x)} \left(1+\frac{1}{\gamma_{0}(\mathbf x)} \right) \esp\left[ \frac{Z}{1+Z/u}\mid \mathbf X \in \mathcal T_\ell \right] \\
\esp\left[h_{\theta_{0}(\mathbf x)}(Z) \mid  \mathbf X \in \mathcal T_\ell\right] &=&-\frac{1}{\gamma_{0}(\mathbf x)^2}\esp\left[ \log(1+Z/u) \mid \mathbf X \in \mathcal T_\ell \right] \\
&&+ \frac{1}{u\gamma_{0}(\mathbf x)}\left(1+\frac{1}{\gamma_{0}(\mathbf x)} \right) \esp\left[ \frac{Z}{1+Z/u}\mid  \mathbf X \in \mathcal T_\ell\right] \,. 
\end{eqnarray*}

Under Assumption \ref{assum:rv:second}, we have
\begin{equation*}
\overline{F}_{u}(z) = \left( 1+\frac{z}{u}\right)^{-1/\gamma_{0}(\mathbf x)} \left \{1 + c\psi(u) \int_1^{1+z/u} v^{\rho-1} \mathd v + o(\psi(u))\right\} \,.
\end{equation*}
\begin{eqnarray*}
\esp\left[ \frac{Z}{1+Z/u}\mid  \mathbf X \in \mathcal T_\ell\right] 
&=& \int_0^u \overline{F}_u \left(\frac{t}{1-t/u} \right) \mathd t\\
&=& \frac{u}{1+1/\gamma_0(\mathbf x)} \left(1 + \frac{ c\psi(u)}{1+1/\gamma_{0}(\mathbf x)-\rho} + o(\psi(u)) \right)\\
&\leq& u \left( 1 + c\gamma_{0}(\mathbf x)\psi(u) + o(\psi(u)) \right)
\end{eqnarray*}

and then
\begin{eqnarray*}
\esp\left[ \log(1+Z/u) \mid  \mathbf X \in \mathcal T_\ell \right]
&=& \int_0^u \mathbb{P}\left[Z \geq u(\mathe^t-1) \mid  \mathbf X \in \mathcal T_\ell  \right] \mathd t \\
&=& \gamma_{0}(\mathbf x)\left(1 + \frac{c\psi(u)}{1/\gamma_{0}(\mathbf x)-\rho}  + o(\psi(u)) \right)\\
&\leq& \gamma_{0}(\mathbf x) \left(1+c\gamma_{0}(\mathbf x)\psi(\mathbf x)(u) + o(\psi(u))  \right) \, .
\end{eqnarray*}

Consequently, 
\begin{eqnarray*}
|\esp\left[g_{\theta_0(\mathbf x)}(Z) \mid  \mathbf X \in \mathcal T_\ell \right] | \leq \frac{1}{\gamma_{\min}}\left(1+\frac{1}{u}\left(1+\frac{1}{\gamma_{\min}}\right) \right) \left( 1 + c\gamma_{0}(\mathbf x) \psi(u) + o(\psi(u)) \right)
\end{eqnarray*}
and 

\begin{eqnarray*}
|\esp\left[h_{\theta_0(\mathbf x)}(Z) \mid  \mathbf X = \mathbf x\right]| \leq \frac{1}{\gamma_{\min}}\left(1 +\frac{1}{\gamma_{\min}}+\frac{\gamma_{\max}}{\gamma_{\min}}\right)  \left( 1 + c\gamma_0(\mathbf x) \psi(u ) + o(\psi(u)) \right) \, . 
\end{eqnarray*}
Hence, 
\begin{eqnarray*}
\|\theta_{0}(\mathbf x) - \theta^*_\ell\|_\infty
&\leq& \mathfrak C_2(u)\frac{k_n}{n}  \left( 1 + c\gamma_{\max} \psi(u) + o(\psi(u)) \right) \,,
\end{eqnarray*}
where $\mathfrak C_2(u)=\frac{1}{\mathfrak C_1}\frac{1}{\gamma_{\min}}\max\left(1+\frac{1}{u}+\frac{1}{u\gamma_{\min}},1+\frac{1}{\gamma_{\min}}+\frac{\gamma_{\max}}{\gamma_{\min}} \right)$.

Finally, 
\begin{eqnarray*}
\|\theta^*(\mathbf x) - \theta_0(\mathbf x)\|_\infty  &\leq& \sum_{\ell=1}^{K_{\max}}\|\theta^*_\ell - \theta_{0}(\mathbf x) \|_\infty \mathbf 1_{\mathbf x \in \mathcal T_\ell}\\
&\leq& \mathfrak C_2(u)\frac{k_n}{n}  \left( 1 + c\gamma_{\max} \psi(u) + o(\psi(u)) \right)\sum_{\ell=1}^{K_{\max}} \mathbf 1_{\mathbf x \in \mathcal T_\ell}\\
&\leq& \mathfrak C_2(u)\frac{k_n}{n}  \left( 1 + c\gamma_{\max} \psi(u) + o(\psi(u)) \right) \, .
\end{eqnarray*}

\subsection{Proof of Theorem \ref{thselection}}
\label{sec3}

The following lemma will be needed to prove Theorem \ref{thselection}. 
\begin{lemma}
\label{pk} Let $\mathfrak D = \inf_u\inf_{K < K_0(u)} \Delta L(T^*(u),T^*_K(u))$ and $ u \in [u_{\min},u_{\max}]$ fixed. Suppose that there exists a constant $c_2>0$ such that the penalization constant $\lambda$ satisfies 
\[
 c_2  \{\log k_n\}^{1/2} k_n^{-1/2} \leq \lambda
\leq(\mathfrak{D} - 2c_2 \{\log(k_n)\}^{1/2} k_n^{-1/2})k_n^{-1},
\]
then, for $K> K_0(u),$
\begin{eqnarray*}
\mathbb{P}(\widehat{K}(u)=K) &\leq & 2\left(\exp\left(- \frac{ { C_1}  k_n \lambda^2(K-K_0(u))^2}{\beta^2 (\log k_n)^{2}} \right) + \exp\left(-\frac{{ C_2} k_n \lambda(K-K_0(u)))}{\beta \log k_n} \right)\right) \\
&& +\frac{{ C_3}}{k_n^{5/2} \lambda^3(K-K_0(u))^3},
\end{eqnarray*}
and, for $K<K_0(u),$
\begin{eqnarray*}
\mathbb{P}(\widehat{K}(u)=K) &\leq & 4\exp\left(- \frac{ C_1  k_n \{\mathfrak{D}-\lambda(K_0(u)-K)\}^2}{\beta^2 (\log k_n)^{2}} \right) \\
&&+4 \exp\left(-\frac{C_2 k_n \{\mathfrak{D}-\lambda(K_0(u)-K)\}}{\beta \log k_n} \right)\\
&& +\frac{2C_3}{k_n^{5/2} \{\mathfrak{D}-\lambda(K_0(u)-K)\}^3}.
\end{eqnarray*}
\end{lemma}

 \begin{proof}

Let  $u \in [u_{\min},u_{\max}]$ fixed. If $\widehat{K}(u)=K,$ this means that
\[
\Delta L_n(T_K(u),T_{K_0(u)}(u)): = L_n(T_K,u)-L_n(T_{K_0(u)},u)>\lambda (K-K_0(u)      ).
\]
Decompose  
\begin{eqnarray*}
\Delta L_n(T_K(u),T_{K_0}(u))&=&\{L_n(T_K,u)-L_n(T^*_K,u)\}+\{L_n(T^*_K,u)-L_n(T^*,u)\}\\ &&+\{L_n(T^*,u)-L_n(T_{K_0(u)},u)\}.
\end{eqnarray*}
Since $L_n(T^*,u)-L_n(T_{K_0(u)},u)<0,$
\[
\Delta L_n(T_K(u),T_{K_0(u)}(u))\leq \{L_n(T_K,u)-L_n(T^*_K,u)\}+\{L_n(T^*_K,u)-L_n(T^*,u)\}.
\]
For $K > K_0(u),$ $T^*_K(u)=T^*(u),$ hence, 
\begin{eqnarray*}
\mathbb{P}(\widehat{K}(u)=K)&\leq &\mathbb{P}\left(\Delta L_n(T_K(u), T^*_K(u))>\lambda (K-K_0(u))\right)\\
&\leq &\mathbb{P}\left( \left|\Delta L_n(T_K(u), T^*_K(u)) - \Delta L( T_K(u), T^*_K(u))\right|>\lambda (K-K_0(u))\right).
\end{eqnarray*}
For $K>K_0(u)$, a bound is then obtained from Theorems \ref{th1} and \ref{th2} if $\lambda(K-K_0(u)) \geq {c_1}  \{\log(k_n)\}^{1/2} k_n^{-1/2}$, that is $\lambda\geq {c_1}  \{\log k_n\}^{1/2} k_n^{-1/2} $.

Now, for $K<K_0(u),$
\begin{eqnarray*}
\Delta L_n(T^*_K(u),T^*(u)) &\leq&  |\Delta L_n(T^*_{K}(u),T^*(u))-\Delta L(T^*_{K}(u),T^*(u))|+\Delta L(T^*_{K}(u),T^*(u))\\
& \leq & |\Delta L_n(T^*(u),T^*_{K}(u))-\Delta L(T^*(u),T^*_K(u))| - \mathfrak D(K_0(u),K).
\end{eqnarray*}
where $\mathfrak D = \inf_{K<K_0(u), u\in[u_{\min}, u_{\max}]} \mathfrak D(K_0(u),K),$ 
Hence, 
\begin{eqnarray*}
\lefteqn{\mathbb{P}(\widehat{K}(u)=K)}\\
 &\leq &\mathbb{P}\left(\Delta L_n(T_K(u), T^*_K(u))\geq \frac{\mathfrak D- \lambda(K_0(u)-K)}{2}\right) \\
&& + \mathbb{P}\left(|\Delta L_n(T^*(u),T^*_K(u))-\Delta L(T^*(u),T^*_K(u))|\geq \frac{\mathfrak D - \lambda(K_0(u)-K)}{2}\right) \\
&\leq &\mathbb{P}\left(\left|\Delta L_n(T_K(u), T^*_K(u))-\Delta L( T_K(u), T^*_K(u))\right|\geq \frac{\mathfrak D - \lambda(K_0(u)-K)}{2}\right) \\
&& + \mathbb{P}\left(|\Delta L_n(T^*(u),T^*_K(u))-\Delta L(T^*(u),T^*_K(u))|\geq \frac{\mathfrak D - \lambda(K_0(u)-K)}{2}\right).
\end{eqnarray*}
These two probabilities can be bounded using Theorems \ref{th1} and \ref{th2} provided that, for all $K<K_0(u),$
\[
\frac{\mathfrak D - \lambda(K_0(u)-K)}{2} \geq \mathfrak c_1 \{\log(k_n)\}^{1/2} k_n^{-1/2},
\]
that is, 
\[
\lambda \leq  \mathfrak D - 2{\mathfrak c_1} \{\log(k_n)\}^{1/2} k_n^{-1/2} . 
\]
\end{proof}

We are now ready to prove Theorem \ref{thselection}. Let $u \in [u_{\min} , u_{\max}]$ fixed.
\begin{eqnarray*}
\esp \left[ \|\widehat{T}(u)-T^*(u)\|_{2}^2\right]&=&\sum_{K=1}^{K_{\max}}\esp\left[\|T_K(u)-T^*(u)\|_2^2\mathbf{1}_{\widehat{K}(u)=K}\right]\\
&\leq& \esp\left[\|T_{K_0(u)}(u)-T^*(u)\|_2^2\right] + \sum_{K=1, K\neq K_0(u)}^{K_{\max}}K \mathbb{P}(\widehat{K}(u)=K) \\
&& +\sum_{K=1, K\neq K_0(u)}^{K_{\max}}\mathbb{E}\left[\|T_K(u)-T^*(u)\|_2^2 \mathbf{1}_{\|T_K(u)-T^*(u)\|_2^2> K}\mathbf{1}_{\widehat{K}(u)=K}\right]\\
&\leq & \esp\left[\|T_{K_0(u)}(u)-T^*(u)\|_2^2\right] + \sum_{K=1}^{K_0(u)-1}K \mathbb{P}(\widehat{K}(u)=K)\\
&& + \sum_{K=K_0(u)+1}^{K_{\max}}K \mathbb{P}(\widehat{K}(u)=K)\\
&& +2\sum_{K=1, K\neq K_0(u)}^{K_{\max}}\mathbb{E}\left[\|T_K(u)-T^*_K(u)\|_2^2\mathbf{1}_{\|T_K(u)-T^*(u)\|_2^2> K}\right]\\
&&+2\sum_{K=1, K\neq K_0(u)}^{K_{\max}}\mathbb{P}(\widehat{K}(u)=K) \|T^*(u)-T_K^*(u)\|_2^2 .
\end{eqnarray*}
Firstly, from Theorem \ref{thm:deviation:bounds},
\begin{eqnarray*}
\lefteqn{\mathbb{E}\left[\|T_K(u)-T_K^*(u)\|_2^2 \mathbf{1}_{\|T_K(u)-T^*(u)\|_2^2> K}\right]}\\
 & = & K \mathbb{P}\left(\|T_K(u)-T_K^*(u)\|_2^2 > K\right) + \int_K^{\infty} \mathbb{P}\left(\|T_K(u)-T_K^*(u)\|_2^2 > t\right) \mathd t\\
  &\leq& 2 K\left(1+\frac{\beta^2 (\log k_n)^2}{\mathcal C_1 k_n}\right)\exp\left(-\frac{\mathcal C_1 k_n }{ \beta^2 (\log k_n)^2}\right)\\
  && +2K\left( 1+ \frac{2\beta (\log k_n)}{\mathcal C_2 k_n}+\frac{2\beta^2 (\log k_n)^2}{\mathcal C_2^2 k_n^2}\right)\exp\left(-\frac{\mathcal C_2 k_n}{\beta (\log k_n)}\right) + \frac{2\mathcal C_3 K^{1/2}}{k_n^{5/2} }\, .
\end{eqnarray*}
Secondly, recall that 
\[
\|T^*_K(u)-T^*(u)\|^2_2 = \int \|\theta^{K*}(\mathbf x)-\theta^{*}(\mathbf x)\|^2_{\infty}\mathd \mathbb P(\mathbf x) \leq K_{\max} \sum_{\ell=1}^{K_{\max}}\|\mu(\mathcal T_\ell)\theta^{K*}_\ell-\theta^{*}_\ell\|_\infty^2  \, ,
\]
where $\mu(\mathcal T_\ell) = \mathbb{P}(\mathbf X \in \mathcal T_\ell)$. 
Following the same idea as in the proof of Proposition  \ref{prp:bias}, from Taylor's expansion, under Assumptions \ref{a_moche} and \ref{assum:rv:second},

\begin{eqnarray*}
\|\theta^{K*}_{\ell} - \theta^*_\ell\|_\infty^2  
& \leq & \mathfrak C^2_2(u)\frac{k_n^2}{n^2}  \left( 1 + c\gamma_{\max} \psi(u) + o(\psi(u)) \right)^2 \,.
\end{eqnarray*}

Hence, 
\begin{eqnarray*}
\|T^*_K(u)-T^*(u)\|_2^2
&\leq& \mathfrak C^2_2(u)\frac{k_n^2}{n^2}  (1 + c\gamma_{\max} \psi(u) + o(\psi(u)))^2\sum_{\ell=1}^{K_{\max}}  \mathbf 1_{\mathbf x \in \mathcal T_\ell}\\
&\leq& \mathfrak C_3(u)\frac{k_n^2}{n^2} \,.
\end{eqnarray*}

Finally,  
\begin{eqnarray*}
\esp\left[\|\widehat{T}(u)-T^*(u)\|_{2}^2 \right] &\leq & \frac{\mathcal{C}_5  K_0(u) (\log k_n)^2 }{k_n},
\end{eqnarray*}
for some constant $\mathcal{C}_5.$ .

\section{Covering numbers}
\label{sec_covering}

\begin{lemma}
\label{lem:cov:numbers}
Following the notations of the proof of Theorem \ref{th1}, the class of functions $\mathfrak{F}$ satisfies
$$\mathcal{N}_{\Phi}(\varepsilon,\mathfrak{F})\leq \frac{\mathfrak C_4K^{4(d+1)(d+2)}\|\Phi \|_2^{\alpha_1}\sigma_{n}^{\alpha}}{\varepsilon^{\alpha}},$$
for some constants $\mathfrak C_4>0$ and $\alpha>0$ (not depending on $n$ nor $K$). 
\end{lemma}
\begin{proof}
Let
\begin{eqnarray*}
g_{\theta}(z) &=& -\frac{1}{\sigma}+\left(\frac{1}{\gamma}+1\right)\frac{\gamma z}{\sigma^2(1+\frac{z\gamma}{\sigma})}, \\
h_{\theta}(z) &=& -\frac{1}{\gamma^2}\log\left(1+\frac{z\gamma}{\sigma} \right)+\frac{\left(\frac{1}{\gamma}+1\right)z}{\sigma+z\gamma},
\end{eqnarray*}
for $z>0.$
For $\theta$ and $\theta'$ in $\mathcal{S}\times \Gamma,$ we have (from a straightforward Taylor expansion),
\[
 |g_\theta(y-u) - g_{\theta'}(y-u)|  \leq  C |\gamma - \gamma'| + C'|\sigma-\sigma'|,
\]
for some constants $C$ and $C'.$ More precisely, one can take 
\begin{eqnarray*}
C &=& \frac{6}{\gamma_{\min}^2\sigma_{\min}},\\
C' &=& \frac{1}{\sigma_{\min}^2}\left(1+3\left\{1+\frac{1}{\gamma_{\min}}\right\}\right).
\end{eqnarray*} Next, observe that
$$|g_{\theta'}(y-u) - g_{\theta'}(y-u')|\leq C''|u-u'|,$$
where $C''=4\gamma^2_{\max}/[\gamma_{\min}\sigma^3].$
Which leads to
$$ |g_\theta(y-u) - g_{\theta'}(y-u')|  \leq  C_g \max(\|\theta-\theta'\|_{\infty},|u-u'|),$$
for some constant $C_g>0.$
Similarly,
\[
|h_\theta(y-u) - h_{\theta'}(y-u)|  \leq C_1(4+\log(1+wy))|\gamma - \gamma' | + C_2|\sigma-\sigma'|,
\]
 Next,
$$|h_{\theta'}(y-u) - h_{\theta'}(y-u')|\leq  C_7 |u-u'|,$$
where $C_7=5/(\gamma_{\min}\sigma_{\min}),$
leading to, for some $C_h>0,$
$$|h_\theta(y-u) - h_{\theta'}(y-u')|\leq C_h \max(\|\theta-\theta'\|_{\infty},|u-u'|).$$

On the other hand,
$$|\phi(y-u,\theta)-\phi(y-u,\theta')|\leq\frac{1}{\gamma_{\min}^2}(2+\log(1+wy))|\gamma-\gamma'|+ \frac{3}{\gamma_{\min}\sigma_{\min}}|\sigma-\sigma'|,$$
and 
$$|\phi(y-u,\theta')-\phi(y-u',\theta')|\leq \frac{1}{\sigma_{\min}}|u-u'|.$$

Define $\mathfrak{F}_1=\{g_{\theta}(\cdot-u):\theta \in \mathcal{S}\times \Gamma, u\in [u_{\min},u_{\max}]\},$ $\mathfrak{F}_2=\{h_{\theta}(\cdot-u):\theta \in \mathcal{S}\times \Gamma, u\in [u_{\min},u_{\max}]\},$ and $\mathfrak{F}_3=\{\phi(\cdot-u,\theta):\theta \in \mathcal{S}\times \Gamma, u\in [u_{\min},u_{\max}]\}.$ From Example 19.7 in \cite{vandervaart}, 
we get, for $i=1,...,3,$
$$N(\varepsilon,\mathfrak{F}_i)\leq F_i \|\Phi \|_2^{\alpha_1}\sigma_{n}^{\alpha_1}\varepsilon^{-\alpha_1},$$
for some $\alpha>0$ and constants $F_i.$

On the other hand, let
$$\mathfrak{F}_4=\left\{\mathbf x \mapsto \mathbf{1}_{\mathbf x \in \mathcal T_\ell} \colon \ell =1,\ldots,K \right\},$$
and
$$\mathfrak{F}_5=\left\{y \mapsto \mathbf{1}_{y>u} \colon u \in \mathcal U \right \}.$$
From Lemma 4 in \citep{lopez2016}, we have $N(\varepsilon,\mathfrak{F}_4)\leq m^k K^{\alpha_2}\varepsilon^{-\alpha_2},$ where $\alpha_2=4(d+1)(d+2),$ and where $k$ is the number of discrete components taking at most $m$ modalities. On the other hand, from Example 19.6 in \citep{vandervaart}, $N(\varepsilon,\mathfrak{F}_5)\leq 2\varepsilon^{-2}.$

From Lemma A.1 in \citep{einmahl2005uniform}, we get, for $i=1,\ldots,3,$
$$N(\varepsilon,\mathfrak{F}_i\mathfrak{F}_4\mathfrak{F}_5)\leq \frac{4 m^kK^{\alpha_2}\max(C_g,C_h)\|\Phi \|_2^{\alpha_1}\sigma_{n}^{\alpha_1}}{\varepsilon^{\alpha_1+\alpha_2+\alpha_3}}.$$
Multiplying $\mathfrak{F}_i\mathfrak{F}_4\mathfrak{F}_5$ by a single indicator function $\mathbf{1}_{\Phi(Y_i)\leq M_n}$ does not change the covering number, and the result follows.
\end{proof}

\section{Technical Lemmas}
\label{sec4}

\begin{lemma}
\label{lemma_vphi}
With $v_{\mathfrak{F}}$ defined in Proposition \ref{talagrand},
$$v_{\mathfrak{F}}\leq \frac{M_n^2k_n}{n}.$$
\end{lemma}

\begin{proof}
We have
\begin{eqnarray*}
v_{\mathfrak{F}} &\leq&  \mathbb{E}\left[\Phi(Y)^2 \mathbf{1}_{Y\geq u_{\min}}\mathbf{1}_{\Phi(Y)\leq M_n}\right] \\
&\leq & M_n^2 \mathbb{P}(Y\geq u_{\min})=\frac{M_n^2k_n}{n}.
\end{eqnarray*}
 \end{proof}

\begin{lemma}
\label{lemma_moments}
Define, for $j=1,2,3,$
$$E_{j,n}=\mathbb{E}\left[\Phi(Y)^j \mathbf{1}_{\Phi(Y)\geq M_n}\mathbf{1}_{Y\geq u_{\min}}\right].$$
Under the assumptions of Theorem \ref{th2}, $$E_{j,n}\leq \frac{\mathfrak{e}_j k_n^{1/2}}{n^{1/2}n^{\rho_0\beta a_2/4}}.$$
\end{lemma}

\begin{proof}
Applying twice Cauchy-Schwarz inequality leads to
$$E_{j,n}\leq \mathbb{P}(Y\geq u_{\min})^{1/2}\esp [\Phi(Y)^{2j}\mathbf{1}_{\Phi(Y)\geq M_n}]^{1/2}\leq \frac{k_n^{1/2}}{n^{1/2}}\esp [\Phi(Y)^{4j}]^{1/4}\mathbb{P}(\Phi(Y)\geq M_n)^{1/4}.$$
Next, from Chernoff inequality,
\begin{equation} \nonumber \mathbb{P}(\Phi(Y)\geq M_n)\leq \exp(-\rho_0 M_n)\esp [\exp(\rho_0 \Phi(Y))]\leq \frac{m_{\rho_0}}{n^{\rho_0\beta a_2}}.
\end{equation}
\end{proof}

\bibliographystyle{abbrvnat}
\bibliography{bibli}
\end{document}